\documentclass[12pt,reqno]{amsart}
\usepackage{amsthm,amsfonts,amssymb,euscript}

\newcommand{\med}{\ensuremath{\text{med}}}
\newcommand{\e}{\ensuremath{\mathbf{e}}}

\newcommand{\R}{{\mathbb R}}

\theoremstyle{plain}
  \newtheorem{theorem}[subsection]{Theorem}
  
  \newtheorem{lemma}[subsection]{Lemma}
  \newtheorem{corollary}[subsection]{Corollary}

\theoremstyle{remark}
  \newtheorem{remark}[subsection]{Remark}

\theoremstyle{definition}

\numberwithin{equation}{section}

\begin{document}
\title[Ishimori equation]{Local Well-posedness for hyperbolic-elliptic Ishimori equation}
\author{Yuzhao Wang}
\address{LMAM, School of Mathematical Sciences, Peking University, Beijing
100871, China}

\email{wangyuzhao2008@gmail.com}

\begin{abstract}
In this paper we consider the hyperbolic-elliptic Ishimori
initial-value problem with the form:
\begin{eqnarray*}
\left \{\begin{array}{ll} \displaystyle \partial_t s
=s\times\square_{x}s+b(\phi_{x_1}s_{x_2}+\phi_{x_2}s_{x_1})\text{ on
}\mathbb{R}^{2}\times
[-1,1];\\
\Delta\phi=2s\cdot(s_{x_1}\times s_{x_2})\\
s(0)=s_{0}
\end{array}
\right.
\end{eqnarray*}
where $s(\cdot,t): \mathbb{R}^{2}\rightarrow \mathbb{S}^{2}\subset
\mathbb{R}^3$, $\times$ denotes the wedge product in
$\mathbb{R}^{3}$, $\square_x =
\partial^{2}_{x_1}-
\partial^{2}_{x_2}$, $b\in \mathbb{R}$. We prove that such system is locally
well-posed for small data $s_{0}\in
H_{Q}^{\sigma_{0}}(\mathbb{R}^{2}; \mathbb{S}^{2})$, $\sigma_{0}>
3/2$, $Q\in \mathbb{S}^{2}$. The new ingredient is that we develop
the methods of Ionescu and Kenig \cite{IK} and \cite{IK2} to
approach the problem in a perturbative way.
\end{abstract}
\maketitle\tableofcontents


\maketitle

\section{Introduction}

In this paper we consider the hyperbolic-elliptic Ishimori equation,
which is an integrable topological spin field model, with the form
\begin{eqnarray}\label{EQ1}
\left \{\begin{array}{ll} \displaystyle \partial_t s
=s\times\square_{x}s+b(\phi_{x_1}s_{x_2}+\phi_{x_2}s_{x_1});\\
\Delta\phi=2s\cdot(s_{x_1}\times s_{x_2})\\
s(0)=s_{0},
\end{array}
\right.
\end{eqnarray}
where $\square_x =
\partial^{2}_{x_1}-
\partial^{2}_{x_2}$, and $s: \mathbb{R}^{2}\times\mathbb{R}\hookrightarrow
 \mathbb{S}^{2}\hookrightarrow \mathbb{R}^{3}$, $\lim_{|x_1|,|x_2|\rightarrow
  \infty}s(x_1,x_2,t)=(0,0,-1)$, $b\in \mathbb{R}$. In \cite{I},
Ishimori introduced the system \eqref{EQ1} in analogy with 2D CCIHS
chain, as a model having the same topological properties as the
latter yet permitting topological vortices whose dynamics were
integrable. Ishimori system \eqref{EQ1} also describes the evolution
of a system of static spin vortices in the plane. Furthermore,
\eqref{EQ1} is completely integrable when $b=1$.

During the past decades, Ishimori system \eqref{EQ1} was widely
studied, see \cite{BIK2,KN,S} and references therein. In \cite{S},
Soyeur proved that the Ishimori system \eqref{EQ1} was well-posed in
$H^m(\R^2)$ for $m \geq 4$. In \cite{KN}, Kenig and Nahmod proved
that the Ishimori system \eqref{EQ1} admited a local in time
solution with large data in the Sobolev class $
H^{m}(\mathbb{R}^2)$, $m> 3/2$, and uniqueness in
$H^{2}(\mathbb{R}^2)$. Recently, for the completely integrable case
$b=1$, Bejenaru, Ionescu and Kenig proved in \cite{BIK2} that the
Ishimori system \eqref{EQ1} was globally well-posed with small data
in the critical Sobolev space
$\dot{H}_Q^{1}(\mathbb{R}^2;\mathbb{S}^2)$.

In this paper, we prove a local well-posedness result for the
Ishimori equation \eqref{EQ1} with $b\in \R$, when the data is small
in Sobolev space $H_Q^{\sigma_0}(\mathbb{R}^2;\mathbb{S}^2)$, for
$\sigma_{0}> 3/2$, $Q\in\mathbb{S}^2$, and the constants in this
paper will dependent on $b$. We begin with some notations. For
$\sigma\geq 0$, let $J^{\sigma}$ denote the operator defined by
Fourier multiplier $(1+|\xi|^{2})^{\sigma/2}$, and $H^{\sigma}$
denote the usual Sobolev space on $\mathbb{R}^{2}$ with the norm
 $\|f\|_{H^{\sigma}}=\|J^{\sigma}(f)\|_{L^{2}}$. Then for
$\sigma \geq 0$ and $Q=(Q_{1},Q_{2},Q_{3})\in \mathbb{S}^{2}$, we
can define the metric space
$$
H^{\sigma}_{Q}=H^{\sigma}_{Q}(\mathbb{R}^{2};\mathbb{S}^{2})=\{f:
\mathbb{R}^{2}\rightarrow \mathbb{S}^{2}: |f(x)|\equiv1 \text{ and }
f_{l}-Q_{l}\in H^{\sigma}\text{ for } l=1,2,3\}
$$
with the induced distance
$$
d_{Q}(f,g)=\Big[\sum_{l=1}^{3}\|f_{l}-g_{l}\|_{H^{\sigma}}^{2}\Big]^{1/2}.
$$
For $Q\in \mathbb{S}^{2}$, we define the complete metric space
$$
H_{Q}^{\infty}(\mathbb{R}^{2})=\bigcap_{\sigma\geq
0}H_{Q}^{\sigma}(\mathbb{R}^{2})\text{ with the induced metric}.
$$
Let $f_{Q}(x)\equiv Q$, $f_{Q}\in H^{\infty}_{Q}$. For any metric
space $X$, $x\in X$, and $r>0$, denote $B_{X}(x,r)=\{y\in X:
d(x,y)<r\}$. $\mathbb{Z}_{+}=\{0,1,2,\ldots\}$ is the nature number
set. We now state our main result.

\begin{theorem}\label{t1}
(a) Assume $\sigma_{0}> 3/2$ and $Q\in \mathbb{S}^{2}$. Then there
is $\epsilon(\sigma_{0})>0$ such that for any $u_{0}\in
H_{Q}^{\infty}\cap B_{H_{Q}^{\sigma_{0}}}(0, \epsilon(\sigma_{0}))$
there is a unique solution
$$
s=S^{\infty}(s_{0})\in C([-1,1]: H_{Q}^{\infty})
$$
of the initial-value problem \eqref{EQ1}.

(b)The mapping $s_{0}\rightarrow S^{\infty}(s_{0})$ extends uniquely
to a Lipschitz mapping
\begin{eqnarray}\label{t11}
S^{\sigma_{0}} : B_{H_{Q}^{\sigma_{0}}}(0,
\epsilon(\sigma_{0}))\rightarrow C([-1,1]: H_{Q}^{\sigma_{0}})
\end{eqnarray}
where $S^{\sigma_{0}}(s_{0})$ is a weak solution of the
initial-value problem \eqref{EQ1} for any $s_{0}\in
B_{H_{Q}^{\sigma_{0}}}(f_{Q}, \epsilon(\sigma_{0}))$.

(c) For any $\sigma\in \mathbb{Z}_{+}$ we have the local Lipchitz
bound
\begin{eqnarray}\label{t12}
\sup_{t\in[-1,1]}d_{Q}^{\sigma_{0}+\sigma}(S^{\sigma_{0}}(s_{0})(t)-S
^{\sigma_{0}}(s_{0}')(t))\leq C(\sigma_{0},\sigma,
R)d_{Q}^{\sigma_{0}+\sigma}(s_{0},s_{0}')
\end{eqnarray}
for any $R>0$ and $s_{0},s_{0}'\in B_{H_{Q}^{\sigma_{0}}}(0,
\epsilon(\sigma_{0}))\cap B_{H_{Q}^{\sigma_{0}+\sigma}}(0,R)$.
\end{theorem}
We use the stereographic projection to reduce \eqref{EQ1} to a
nonlinear non-elliptic Schrodinger equation \eqref{t21} below. We
refer the readers to \cite{IK} or \cite{S} for the details. Finally,
for Theorem \ref{EQ1} it suffices to prove Theorem \ref{t2} below.
\begin{theorem}\label{t2}
(a) Assume $\sigma_{0}> 3/2$. Then there is $\epsilon(\sigma_{0})>0$
with the property that for any $u_{0}\in H^{\infty}\cap
B_{H^{\sigma_{0}}}(0, \epsilon(\sigma_{0}))$ there is a unique
solution
$$
u=\widetilde{S}^{\infty}(\phi)\in C([-1,1]: H^{\infty})
$$
of the initial-value problem
 \begin{eqnarray}\label{t21}
\left \{\begin{array}{ll} \displaystyle (i\partial_t +\square) u
=\frac{2\bar{u}}{1+u\bar{u}}[(\partial_{x_1}u)^{2}-(\partial_{x_2}u)^{2}]+ib(\phi_{x_1}u_{x_2}+\phi_{x_2}u_{x_1})
\\
\displaystyle\Delta\phi=4i\frac{(u_{x_1}\bar{u}_{x_2}-\bar{u}_{x_1}u_{x_2})}{(1+u\bar{u})^2}
\\
u(0,x_1,x_2)=u_0(x_1,x_2)
\end{array}
\right.
\end{eqnarray}

(b)The mapping $\phi\rightarrow \widetilde{S}^{\infty}(\phi)$
extends uniquely to a Lipschitz mapping
\begin{eqnarray}\label{t11}
\widetilde{S}^{\sigma_{0}} : B_{H^{\sigma_{0}}}(0,
\epsilon(\sigma_{0}))\rightarrow C([-1,1]: H^{\sigma_{0}})
\end{eqnarray}
where $\widetilde{S}^{\sigma_{0}}(\phi)$ is a weak solution of the
initial-value problem \eqref{t21} for any $u_{0}\in
B_{H^{\sigma_{0}}}(0, \epsilon(\sigma_{0}))$.

(c) For any $\sigma'\in \mathbb{Z}_{+}$ we have the local Lipchitz
bound
\begin{eqnarray}\label{t12}
\sup_{t\in[-1,1]}\|\widetilde{S}^{\sigma_{0}}(\phi)(t)-\widetilde{S}
^{\sigma_{0}}(\phi')(t)\|_{H^{\sigma_{0}+\sigma'}}\leq
C(\sigma_{0},\sigma', R)\|\phi-\phi'\|_{H^{\sigma_{0}+\sigma'}}
\end{eqnarray}
for any $R>0$ and $\phi, \phi'\in B_{H^{\sigma_{0}}}(0,
\epsilon(\sigma_{0}))\cap B_{H^{\sigma_{0}+\sigma'}}(0,R)$.
\end{theorem}

\begin{remark}
Furthermore, we can generalize Theorem \ref{t2}. The results in the
Theorem \ref{t2} for system \eqref{t21} also hold for the following
system
 \begin{eqnarray}\label{t22}
\left \{\begin{array}{ll} \displaystyle (i\partial_t +\square) u
=\mathcal{N}(u)+ib(\phi_{x_1}u_{x_2}+\phi_{x_2}u_{x_1})
\\
\displaystyle\Delta\phi=4i\frac{(u_{x_1}\bar{u}_{x_2}-\bar{u}_{x_1}u_{x_2})}{(1+u\bar{u})^2}
\\
u(0,x_1,x_2)=u_0(x_1,x_2)
\end{array}
\right.
\end{eqnarray}
where $\mathcal{N}(u)=F_1(u,\overline{u})Q_1(\nabla u, \nabla
u)+F_2(u,\overline{u})Q_2(\nabla u, \nabla
\overline{u})+F_3(u,\overline{u})Q_3(\nabla \overline{u}, \nabla
\overline{u})$, $F_i(\cdot,\cdot)$, $i=1,2,3$, are analytic
functions, and $Q_i(\cdot,\cdot)$, for $i=1,2,3$, are quadratic
forms. For system \eqref{t22}, there is no null structure in the
nonlinear term, so we can not expect to use the method in
\cite{BIK,BIKT,IK2} to get the regularity in $H^s$, $s<3/2$.
\end{remark}

In this paper, we use the methods of Ionescu and Kenig \cite{IK} for
Schrodinger map equation
\begin{eqnarray}\label{Sch}
\left \{\begin{array}{ll} \displaystyle (i\partial_t +\Delta) u
=\frac{2\bar{u}}{1+u\bar{u}}[(\partial_{x_{1}}u)^{2}+(\partial_{x_{2}}u)^{2}]
\\
u(0,x)=\phi(x).
\end{array}
\right.
\end{eqnarray}
We sketch the proof of our main theorem here. We study \eqref{t21}
in a space with high frequency spaces that have two components: an
$X^{\sigma,b}$-type component measured in the frequency space and a
normalized $L^{1,2}_{\mathbf{e}}$ component measured in the physical
space. Then we set up suitable linear($L^{\infty, 2}_{\mathbf{e}}$
smoothing estimate, $L^{2,\infty}_{\mathbf{e}}$ maximal function
estimate) and nonlinear(trilinear estimate, multilinear estimate)
estimates in these spaces, and conclude the results by a recursive
construction.

However, there are some differences between this paper and
\cite{IK}. Firstly, in view of bilinear estimate, the
$X^{\sigma,b}$-type spaces corresponding to non-elliptic group
$e^{it\square}$ are essentially different form the
$X^{\sigma,b}$-type spaces corresponding to elliptic group
$e^{it\Delta}$, see Onodera \cite{O} for detailed argument.
Secondly, the approach in \cite{IK} for the local smoothing estimate
depends on the elliptic property of the group $e^{it\Delta}$. We
develop another approach to get the local smoothing property, which
is based on the following ingredients: Denote $P(\tau,
\xi)=\tau+\xi_{1}^{2}-\xi_{2}^{2}$,
$\mathbf{e}=(\cos\theta,\sin\theta)\in \mathbb{S}^1$, and
$\overline{\mathbf{e}}=(\cos\theta,-\sin\theta)$, for some
$\theta\in [0,2\pi)$, we have
$$
\eta_k(\xi\cdot \overline{\mathbf{e}})\frac{1}{P(\tau,\xi)}\approx
\eta_k(\xi\cdot
\overline{\mathbf{e}})\frac{1}{2^k(\xi_{1}^\mathbf{e}-t_\e^*(\xi_2^\e,\tau))}.
$$
Which behaves like Hilbert transform in $\xi_{1}^{\e}$ direction,
see Lemma \ref{lp4*} below for further argument. Thirdly, the
estimate
\begin{eqnarray}\label{ap1}
\|\mathcal{F}_{(2+1)}^{-1}(f)\|_{L^{\infty, 2}_{\mathbf{e}}}\leq
C(k+1)\|f\|_{Z_{k}}
\end{eqnarray}
is false for non-elliptic type $Z_k$ space(see Remark \ref{re1}
below), which is the main ingredient in the proof of algebra
property(multilinear estimate) in \cite{IK}. We use bilinear
estimate to overcome this problem, the key point is the identity
(see \eqref{null} below)
\begin{eqnarray}\label{null0}
H(uv)
=(Hu)v+u(Hv)+[\partial_{x_1}u\partial_{x_1}v-\partial_{x_2}u\partial_{x_2}v],
\end{eqnarray}
where $H=(i\partial_t+\Box)$. This idea has first appeared in
\cite{IK2} as far as we knew. Fourthly, when $b\neq 0$, the
potential $\phi$ introduces a nonlocal term, we use the method in
\cite{S} to show that this nonlocal term behaves roughly like the
nonlinear term $(2\bar{u}/(1+|u|^2))(u_x^2-u_y^2)$. Finally, the
semi-group $e^{it\Delta}$ is invariant under rotation of the space,
but $e^{it\Box}$ is not. For example, if we rotate the $x$-axes
clockwise $\pi/4$, then $e^{it\Box}$ becomes to
$e^{it\partial_{x_1}\partial_{x_2}}$. So we need to be more careful
when we rotate the space to transform norm $L^{p, q}_{\mathbf{e}}$
to $L^{p}_{x_1}L^{q}_{x_2,t}$.

The rest of the paper is organized as follows. In Section 2 we
define some notations and the main resolution spaces. In Section 3
we establish some basic estimates. In Section 4 we prove the linear
estimates. In Section 5 and 6 we prove the main nonlinear estimates.
In the last section we prove the main theorem. The key ingredients
in these proofs are $L^{2,\infty}_{\mathbf{e}}$ (maximal function)
estimate in Lemma \ref{lp4}, $L^{\infty,2}_{\mathbf{e}}$ (local
smoothing) estimate in Lemma \ref{lp5} and the algebra property in
Lemma \ref{n2}.

\section{Notations and main resolution spaces}

Let $\eta_0 : \mathbb{R}\rightarrow [0,1]$ be a smooth even function
supported in the set $\{\mu\in \mathbb{R} : |\mu|\leq 2\}$ and equal
to 1 in the set $\{\mu\in \mathbb{R} : |\mu|\leq 1/2\}$. We define
$\eta_k : \mathbb{R}\rightarrow [0,1], k=1,2,\cdots$,
$$
\eta_k(\mu)=\eta_0(\mu/2^k)-\eta_0(\mu/2^{k-1})
$$
and $\eta_{k}^{(2)} : \mathbb{R}^{2}\rightarrow [0,1],
k\in\mathbb{Z}_{+}$, $ \eta_{k}^{(2)}(\mu)=\eta_k(|\mu|)$. The
smooth cut-off functions $\chi_{k,l}$
\begin{eqnarray*}
\left \{\begin{array}{ll}\chi_{k,l}(r)=[1-\eta_{0}(r/2^{k-l})]
\text{ if } k\geq 100
\\
\chi_{k,l}(r)=1\text{ if } k\leq 99.
\end{array}
\right.
\end{eqnarray*}

Now we begin to define the normed spaces $X_{k}$ and $Y_{k}$. The
Fourier transform of the linear part of \eqref{t21} in the original
coordinate is
$$
\mathcal{F}_{(2+1)}[(i\partial_t +\square)
u](\xi,\tau)=-(\tau+\xi_{1}^{2}-\xi_{2}^{2})\mathcal{F}_{(2+1)}(u)(\xi,\tau).
$$
We denote
\begin{eqnarray}\label{P}
P(\tau,\xi)=\tau+\xi_{1}^{2}-\xi_{2}^{2}.
\end{eqnarray}
For $\xi\in \mathbb{R}^2$, $\overline{\xi}$ denotes the vector
conjugate to $\xi$, say, if $\xi=(\xi_1,\xi_2)$, then
$\overline{\xi}=(\xi_1,-\xi_2)$. For $k\in \mathbb{Z}_{+}$, $j\in
\mathbb{Z}_{+}$ and $\mathbf{e}\in \mathbb{S}^1\subset \R^2$, denote
\begin{eqnarray*}
\left \{\begin{array}{ll} D_{k,j}=\{(\xi,\tau)\in
\mathbb{R}^{2}\times\mathbb{R}: |\xi|\in [2^{k-1},2^{k+1}]\text{ and
}|P(\tau,\xi)|\leq 2^{j+1}\};
\\
D_{0,j}=\{(\xi,\tau)\in \mathbb{R}^{2}\times\mathbb{R}:
|\xi|\leq2\text{ and }|P(\tau,\xi)|\leq 2^{j+1}\};
\\
D_{k,\infty}=\cup_{j\geq0}D_{k,j};\\
D_{k,j}^\mathbf{e}=\{(\xi,\tau)\in D_{k,j}:
|\xi\cdot\overline{\mathbf{e}}|\geq |\xi|/2\}\text{ for }
j\in\mathbb{Z}_{+}\text{ and } j=\infty.
\end{array}
\right.
\end{eqnarray*}
We define $X_{k}$ normed spaces by
\begin{eqnarray}\label{X_{k}}
\left \{\begin{array}{ll} X_{k}=\{f\in
L^{2}(\mathbb{R}^{2}\times\mathbb{R}):
\text{supp}f\subset D_{k,\infty}; \|f\|_{X_{k}}<\infty\}\\
\text{where
}\|f\|_{X_{k}}=\sum_{j=0}^{\infty}2^{j/2}\|\eta_{j}(P(\tau,\xi))f\|_{L^{2}}.
\end{array}
\right.
\end{eqnarray}
For any vector $\mathbf{e}\in \mathbb{S}^1$, we write
$\mathbf{e}=(\cos\theta,\sin\theta)$, for some $\theta\in [0,2\pi)$,
and $\mathbf{e}^{\perp}=(-\sin\theta,\cos\theta)\in \mathbb{S}^{1}$
perpendicular to $\mathbf{e}$. For $p,q\in [1, \infty]$ the normed
spaces $L^{p,q}_{\mathbf{e}}=
L^{p,q}_{\mathbf{e}}(\mathbb{R}^{2}\times\mathbb{R})$ is defined by
$$
L^{p,q}_{\mathbf{e}}=\left\{f\in
L^{2}(\mathbb{R}^{2}\times\mathbb{R});\quad
\|f\|_{L^{p,q}_{\mathbf{e}}}=\left[\int_{\mathbb{R}}\left[\int_{\mathbb{R}^{2}}|
f(r\e+v\mathbf{e}^{\perp})|^{q}dvdt\right]^{p/q}dr\right]^{1/p}<\infty\right\}.
$$
For $k\geq 100$ and $\mathbf{e}\in \mathbb{S}^1$ we define the
$Y_{k}^\e$ normed spaces
\begin{eqnarray}\label{Y_{k}}
\left \{\begin{array}{ll} Y^\mathbf{e}_{k}=\{f\in
L^{2}(\mathbb{R}^{2}\times\mathbb{R}): \text{supp} f \subset
D^\mathbf{e}_{k,\infty},\quad  \|f\|_{Y^\mathbf{e}_k}<
\infty\} \\
\text{where
}\|f\|_{Y^\mathbf{e}_k}=2^{-k/2}\|\mathcal{F}^{-1}[(P(\tau,\xi)+i)\cdot
f]\|_{L^{1,2}_{\mathbf{e}}}.
\end{array}
\right.
\end{eqnarray}
For the cases $k=0,1,2,\ldots,99$, $Y_{k}^\mathbf{e}=\{0\}$.

Fix $L$ large enough, let $\{\mathbf{e}_{l}\}_{l=1}^{L}\subset
\mathbb{S}^{1}$ satisfy
\begin{enumerate}
                  \item $\mathbf{e}_{l}\neq \mathbf{e}_{l'}$ if $l\neq
                  l'$,
                  \item for any $\mathbf{e}\in \mathbb{S}^{1}$ there is $\mathbf{e}_l$
          such that $|\mathbf{e}_{l}-\mathbf{e}|\leq 2^{-50}$,
                  \item if $\mathbf{e}\in
\{\mathbf{e}_{l}\}_{l=1}^{L}$, then
$-\mathbf{e}\in\{\mathbf{e}_{l}\}_{l=1}^{L}$.
\end{enumerate}
For $k\in \mathbb{Z}_{+}$, we define
\begin{eqnarray}\label{Z_{k}}
Z_{k}=X_{k}+Y_{k}^{\mathbf{e}_{1}}+\ldots +Y^{\mathbf{e}_{L}}_{k}
\end{eqnarray}

\section{Preliminary lemmas}

In this section we prove some preliminary lemmas which will be used
frequently in the following sections.
\begin{lemma}[Spaces Decomposition]\label{lp1}
For $f\in Z_{k}$, in view of the definitions, we can write
\begin{eqnarray}\label{lp11}
\left \{\begin{array}{ll}
f=\sum_{j\in\mathbb{Z}_{+}}g_{j}+f_{\mathbf{e}_{1}}+\ldots
+f_{\mathbf{e}_{L}}\text{ where supp}g_{j}\subset D_{k,j},\text{ and } f_{\mathbf{e}_{l}}\in Y_{k}^{\mathbf{e}_{l}}\\
\sum_{j\in\mathbb{Z}_{+}}2^{j/2}\|g_{j}\|_{L^{2}}+\|f_{\mathbf{e}_{1}}\|_{Y_{k}^{\mathbf{e}_{1}}}+\ldots
+\|f_{\mathbf{e}_{L}}\|_{Y_{k}^{\mathbf{e}_{L}}}\leq 2\|f\|_{Z_{k}}.
\end{array}
\right.
\end{eqnarray}
\end{lemma}
\begin{lemma}\label{lp2}
Fix $\theta\in[0,2\pi)$, $\mathbf{e}=(\cos\theta,\sin\theta)$,
$\mathbf{e}^{\perp}=(-\sin\theta,\cos\theta)$,
$\overline{\mathbf{e}}=(\cos\theta,-\sin\theta)$. Let
$\xi=(\xi_1,\xi_2)$ be in the original coordinate. If we write
$\xi=\xi_{1}^\mathbf{e}\mathbf{e}+\xi_{2}^\mathbf{e}\mathbf{e}^{\perp}$,
 then
\begin{eqnarray}\label{lp22}
\partial_{\xi_{1}^\mathbf{e}} P(\tau,\xi)=\partial_{\xi_{1}^\mathbf{e}}(\xi_{1}^{2}-\xi_{2}^{2}) =2\xi\cdot
\overline{\mathbf{e}},
\end{eqnarray}
where $P(\tau,\xi)$ is defined in \eqref{P}.
\end{lemma}
\begin{proof}[\textbf{Proof of Lemma \ref{lp2}}]
First from $\xi=(\xi_{1},\xi_{2})$ and
$\xi=\xi_{1}^\mathbf{e}\mathbf{e}+\xi_{2}^\mathbf{e}\mathbf{e}^{\perp}$,
we have
\begin{eqnarray}\label{co}
\left \{\begin{array}{ll} \displaystyle \xi^\mathbf{e}_{1}=\xi\cdot
\mathbf{e}=\xi_{1}\cos\theta+\xi_{2}\sin\theta
\\
\xi^\mathbf{e}_{2}=\xi\cdot
\mathbf{e}^{\perp}=-\xi_{1}\sin\theta+\xi_{2}\cos\theta
\end{array}
\right. \text{ and so } \left \{\begin{array}{ll} \displaystyle
\xi_{1}=\xi^\mathbf{e}_{1}\cos\theta-\xi^\mathbf{e}_{2}\sin\theta
\\
\xi_{2}=\xi^\mathbf{e}_{1}\sin\theta+\xi^\mathbf{e}_{2}\cos\theta
\end{array}
\right.
\end{eqnarray}
in the new coordinate $(\mathbf{e},\mathbf{e}^{\perp})$, we have
\begin{eqnarray}\label{pco}
P(\tau,\xi)&=&\tau+\xi_{1}^{2}-\xi_{2}^{2}\nonumber\\
&=&\tau+(\xi_{1}^\mathbf{e}\cos\theta-\xi_{2}^\mathbf{e}\sin\theta)^{2}-
(\xi_{1}^\mathbf{e}\sin\theta+\xi_{2}^\mathbf{e}\cos\theta)^{2}
\end{eqnarray}
by a simple calculation, we have
\begin{eqnarray*}
\partial_{\xi_{1}^\mathbf{e}}P(\tau,\xi)
 &=&2(\xi_{1}^\mathbf{e}\cos\theta-\xi_{2}^\mathbf{e}\sin\theta)\cos\theta-
2(\xi_{1}^\mathbf{e}\sin\theta+\xi_{2}^\mathbf{e}\cos\theta)\sin\theta\\
 &=&2\xi_{1}\cos\theta-2\xi_{2}\sin\theta\\
&=&2\xi\cdot \overline{\mathbf{e}}.
\end{eqnarray*}
\end{proof}

\begin{lemma}\label{lp3}(Multipliers)
1.If $m\in L^{\infty}(\mathbb{R}^{2})$,
$\mathcal{F}^{-1}_{(2)}(m)\in L^{1}(\mathbb{R}^{2})$, and $f\in
Z_{k}$, then $m(\xi)\cdot f\in Z_{k}$ and
\begin{eqnarray}\label{lp3a}
 \|m(\xi)\cdot f\|_{Z_{k}}\leq
C\|\mathcal{F}^{-1}_{(2)}(m)\|_{L^{1}(\mathbb{R}^{2})}\cdot\|f\|_{Z_{k}}\hspace{0cm}
\end{eqnarray}
2. If $f\in Z_{k}$, $k,j\in \mathbb{Z}_{+}$, and $C_1\in \mathbb{R}$
is a constant, then
\begin{eqnarray}\label{lp3b}
 \left \{\begin{array}{ll}
\|\eta_{j}(P(\tau,\xi))\cdot f\|_{X_{k}}\leq C\|f\|_{Z_{k}}\\
\|\eta_{>2k+C_1}(P(\tau,\xi))\cdot f\|_{X_{k}}\leq C\|f\|_{Z_{k}}
\\\|f\|_{X_{k}}\leq C(k+1)\|f\|_{Z_{k}}.
\end{array}\hspace{0cm}
\right.
\end{eqnarray}
3. If $f\in Z_{k}$, $k,j\in \mathbb{Z}_{+}$, then
\begin{eqnarray}\label{lp3c}
\|\eta_{\leq j}(P(\tau,\xi))\cdot f\|_{Z_{k}}\leq
C\|f\|_{Z_{k}}\hspace{0cm}
\end{eqnarray}where $P(\tau,\xi)$ is defined in \eqref{P}.
\end{lemma}
\begin{proof}[\textbf{Proof of Lemma \ref{lp3}}]Clearly,
\eqref{lp3a} follows directly from the definition. Now we turn to
\eqref{lp3b}. In view of Lemma \ref{lp1}, we can assume $k\geq 100$,
and $f=f_{\mathbf{e}}\in Y_{k}^\mathbf{e}$. Let
\begin{eqnarray*}
h_{\mathbf{e}}(x,t)=2^{-k/2}\mathcal{F}^{-1}_{(2+1)}[(P(\tau,\xi)+i)\cdot
f_{\mathbf{e}}](x,t),
\end{eqnarray*}
thus
\begin{eqnarray}\label{he}
f_{\mathbf{e}}(\xi,\tau)=\chi_{k,5}(\xi\cdot
\overline{\mathbf{e}})\cdot
\frac{2^{k/2}}{P(\tau,\xi)+i}\mathcal{F}_{(2+1)}(h_{\mathbf{e}})(\xi,\tau)
\end{eqnarray} where $\overline{\mathbf{e}}$ is the same as in Lemma \ref{lp2},
 and $\|f_{\mathbf{e}}\|_{Y_{k}^\mathbf{e}}=\|h_{\mathbf{e}}\|_{L_{\mathbf{e}}^{1,2}}$.
By the definition, for \eqref{lp3b} it suffices to prove that
\begin{eqnarray}\label{es32}
2^{k/2}2^{-j/2}\big\|1_{D_{k,j}}(\xi,\tau) \chi_{k,5}(\xi\cdot
\overline{\mathbf{e}})\cdot\mathcal{F}_{(2+1)}(h)(\xi,\tau)\big\|_{L^{2}_{\xi,\tau}}\leq
C(1+2^{j-2k})^{-1/2}\|h\|_{L_{\mathbf{e}}^{1,2}}
\end{eqnarray}
for any $h\in \mathcal{S}(\mathbb{R}^{2}\times\mathbb{R})$ and $j\in
\mathbb{Z}_{+}$. We write
$\xi=\xi^\mathbf{e}_{1}\e+\xi^\mathbf{e}_{2}\mathbf{e}^{\perp}$,
$x=x^\mathbf{e}_{1}\e+x^\mathbf{e}_{2}\mathbf{e}^{\perp}$, and
\begin{eqnarray*}
h'(x^\mathbf{e}_{1},\xi_{2}^\mathbf{e},\tau)=\int_{\mathbb{R}\times
\mathbb{R}}h(x^\mathbf{e}_{1}\e+x^\mathbf{e}_{2}\mathbf{e}^{\perp},t)e^{-i(x^\mathbf{e}_{2}
\xi^\mathbf{e}_{2}+t\tau)}dx^\mathbf{e}_{2}dt.
\end{eqnarray*}
Thus
\begin{eqnarray*}
\mathcal{F}_{(2+1)}(h)(\xi^\mathbf{e}_{1}\e+
\xi^\mathbf{e}_{2}\mathbf{e}^{\perp},\tau)=\int_{\mathbb{R}}h'(x^\mathbf{e}_{1},\xi_{2}^\mathbf{e},\tau)e^{-ix^\mathbf{e}_{1}
\xi^\mathbf{e}_{1}}dx^\mathbf{e}_{1},
\end{eqnarray*}
and by Plancherel theorem, we get that
$\|h\|_{L^{1,2}_{\mathbf{e}}}=C\|h'\|_{L_{x^\mathbf{e}_{1}}^{1}L_{\xi^\mathbf{e}_{2},\tau}^{2}}$.
Thus, for \eqref{es32} it suffices to prove that
\begin{eqnarray*}
&&2^{(k-j)/2}\Big\|1_{D_{k,j}}\cdot \chi_{k,5}(\xi\cdot
\overline{\mathbf{e}})\cdot
\int_{\mathbb{R}}h'(x_{1}^\mathbf{e},\xi_{2}^\mathbf{e},\tau)e^{-ix_{1}^\mathbf{e}\xi_{1}^\mathbf{e}}dx_{1}^\mathbf{e}
\Big\|_{L^{2}_{\xi_{1}^\mathbf{e},\xi_{2}^\mathbf{e},\tau}}\\&\leq&
C(1+2^{j-2k})^{-1/2}\|h'\|_{L_{x_{1}^\mathbf{e}}^{1}L_{\xi_{2}^\mathbf{e},\tau}^{2}}.
\end{eqnarray*}
By H\"{o}lder's inequality, it suffices to prove that
\begin{eqnarray}\label{*}
|\{\xi_{1}^\mathbf{e}: |\xi|\leq 2^{k}, |\xi\cdot
\overline{\mathbf{e}}|\geq 2^{k-10}\text{ and } |P(\tau,\xi)|\leq
2^{j+1}\}|\leq C\min(2^{j-k},2^{k}).
\end{eqnarray}
This follows easily from \eqref{lp22}.

Now we turn to prove \eqref{lp3c}. In view of \eqref{lp3b}, we can
assume that $k\geq 100$. By the definition, it suffices to prove
that
\begin{eqnarray}\label{lp33}
\|\mathcal{F}_{(2+1)}^{-1}[\eta_{\leq j}(P(\tau,\xi))\cdot f\cdot
\eta_{k}(\xi)\cdot\chi_{k,5}(\xi\cdot
\overline{\mathbf{e}})]\|_{L^{1,2}_{\mathbf{e}}}\leq
C\|\mathcal{F}_{(2+1)}^{-1}(f)\|_{L_{\mathbf{e}}^{1,2}},
\end{eqnarray}
for $j\leq 2k-100$. We write
$\xi=\xi^\mathbf{e}_{1}\e+\xi^\mathbf{e}_{2}\mathbf{e}^{\perp}$,
$x=x^\mathbf{e}_{1}\e+x^\mathbf{e}_{2}\mathbf{e}^{\perp}$. Using
Plancherel theorem and H\"{o}lder's, for \eqref{lp33} it suffices to
show
\begin{eqnarray}\label{lp35}
\left\|\int_{\mathbb{R}}e^{ix^\mathbf{e}_{1}\xi^\mathbf{e}_{1}}\eta_{\leq
j}(P(\tau,\xi))\cdot \eta_{k}(\xi)\cdot\chi_{k,5}(\xi\cdot
\overline{\mathbf{e}})d\xi_{1}^\mathbf{e}\right\|_{L^{1}_{x_{1}^\mathbf{e}}L^{\infty}_{\xi_{2}^\mathbf{e},\tau}}\leq
C.
\end{eqnarray}
In view of \eqref{lp22} and \eqref{pco}, we have that
$|\partial_{\xi_{1}^\mathbf{e}}P(\tau,\xi)|=2|\xi\cdot
\overline{\mathbf{e}}|\leq C2^{k}$ and
$|\partial^{2}_{\xi_{1}^\mathbf{e}}P(\tau,\xi)|=|\cos^2\theta-\sin^2\theta|\leq
1$. From integration by parts and \eqref{*}, we get that
\begin{eqnarray}\label{lp36}
\left|\int_{\mathbb{R}}e^{ix^\mathbf{e}_{1}\xi^\mathbf{e}_{1}}\eta_{\leq
j}(P(\tau,\xi))\cdot \eta_{k}(\xi)\cdot\chi_{k,5}(\xi\cdot
\overline{\mathbf{e}})d\xi_{1}^\mathbf{e}\right|\leq
C\frac{2^{j-k}}{1+(2^{j-k}x_{1}^\mathbf{e})^{2}},
\end{eqnarray}
which gives \eqref{lp35}.
\end{proof}

\begin{lemma}[Representation for $Y_{k}^\mathbf{e}$ functions]\label{lp4*}
If $k\geq 100$, $\mathbf{e}\in \mathbb{S}^{1}$,
$\mathbf{e}=(\cos\theta,\sin\theta)$ and $f\in Y_k^{\e}$ then we can
write
\begin{eqnarray}\label{es2*}
&&f^\mathbf{e}(\xi_{1}^\mathbf{e}\mathbf{e}+\xi_{2}^\mathbf{e}\mathbf{e}^{\perp},\tau)\nonumber\\
&=&2^{k/2}\frac{\eta_{\leq k-100}(\xi_{1}^\mathbf{e}-t^*_\mathbf{e})
\chi_{k,5}(M^*_\e)}{\xi_{1}^\mathbf{e}-t^*_\mathbf{e}+i/2^{k}}\frac{1}{2M^*_\e}
\int_{\mathbb{R}}e^{-iy_{1}^\mathbf{e}\xi_{1}^\mathbf{e}}h(y_{1}^\mathbf{e},
\xi_{2}^\mathbf{e},\tau)dy_{1}^\mathbf{e}+g
\end{eqnarray}
where $\xi^\mathbf{e}_1,\xi_{2}^\mathbf{e},\tau\in\R$, and
$t_\mathbf{e}^{*}=t^*_\mathbf{e}(\xi_{2}^\mathbf{e},\tau)$ satisfies
$
P(\tau,t_\mathbf{e}^{*}\mathbf{e}+\xi_{2}^\mathbf{e}\mathbf{e}^{\perp})=0
$, that is
\begin{eqnarray}\label{es22}
\tau+(t_\mathbf{e}^{*}\cos\theta-\xi_{2}^\mathbf{e}\sin\theta)^{2}-
(t_\mathbf{e}^{*}\sin\theta+\xi_{2}^\mathbf{e}\cos\theta)^{2}=0
\end{eqnarray}
Denote
\begin{eqnarray}\label{mm}
M_\mathbf{e}^{*}=M_\mathbf{e}^{*}(\xi_{2}^\mathbf{e},\tau)=t^*_\mathbf{e}(\cos^2\theta-\sin^2\theta)
-2\xi_{2}^\mathbf{e}\cos\theta\sin\theta
\end{eqnarray}
Furthermore, we have
\begin{enumerate}
\item\label{it:0} $t^*_\e\in \R$ in the support of $f^\e$.
\item\label{it:1} $M_\mathbf{e}^{*}= (t_\e^*\e+\xi_2^\e\e^\perp)\cdot\overline{\e} \sim 2^k$ in the support of $f^\e$.
\item\label{it:2} For the $g$ and $h$ defined above, we
  have
  \begin{equation}\label{pp3}
  ||g||_{X_k}+||h||_{L^1_{y^\mathbf{e}_1}L^2_{\xi_{2}^\mathbf{e},\tau}}\leq
  C ||f^\mathbf{e}||_{Y^\mathbf{e}_k}.
  \end{equation}
\item\label{it:3} In the support of $f^\e$, we have
  \begin{equation}\label{pp4}
  |\partial_{\tau}t^*_\mathbf{e}(\xi_{2}^\mathbf{e},
  \tau)|\geq C2^{-k}.  \end{equation}
\end{enumerate}
\end{lemma}

\begin{proof}[Proof of Lemma \ref{lp4*}] Let
\begin{equation*}
h'(x,t)=2^{-k/2}\mathcal{F}^{-1}[(P(\tau,\xi)+i)\cdot
f^\mathbf{e}](x,t).
\end{equation*}
Thus
\begin{equation}\label{vc21}
\begin{cases}
&f^\mathbf{e}(\xi_{1}^\mathbf{e}\mathbf{e}+\xi_{2}^\mathbf{e}\mathbf{e}^{\perp},\tau)=\chi_{k,5}(\xi\cdot\overline{\mathbf{e}})\cdot
\frac{2^{k/2}}{P(\tau,\xi)+i}\mathcal{F}(h')(\xi_{1}^\mathbf{e}\mathbf{e}+\xi_{2}^\mathbf{e}\mathbf{e}^{\perp},\tau);\\
&\|h'\|_{L^{1,2}_ \e}= C||f^\mathbf{e}||_{Y^\mathbf{e}_k}.
\end{cases}
\end{equation}
Let
\begin{equation*}
h''(y^\mathbf{e}_1,\xi_{2}^\mathbf{e},\tau)=\int_{P_\e\times\R}h'(y_{1}^\mathbf{e}\e+y_{2}^\mathbf{e}\mathbf{e}^{\perp},t)
e^{-i(y_{2}^\mathbf{e}\xi_{2}^\mathbf{e}+t\tau)}\,dy_{2}^\mathbf{e}dt.
\end{equation*} By \eqref{vc21}, we have
\begin{equation}\label{vc211}
\begin{cases}
&f^\mathbf{e}(\xi_{1}^\mathbf{e}\mathbf{e}+\xi_{2}^\mathbf{e}\mathbf{e}^{\perp},\tau)=\chi_{k,5}(\xi\cdot\overline{\mathbf{e}})\cdot
\frac{2^{k/2}}{P(\tau,\xi)+i}\int_{\R}h''(y_1^\e,\xi_{2}^\mathbf{e},\tau)e^{-iy^\mathbf{e}_1\xi^\mathbf{e}_1}\,dy^\mathbf{e}_1;\\
&\|h'\|_{L^{1,2}_
\e}=\|h''\|_{L^{1}_{y_{1}^\mathbf{e}}L^{2}_{\xi_{2}^\mathbf{e},\tau}}.
\end{cases}
\end{equation}
In view of \eqref{lp3b},
\begin{equation*}
||\eta_{\geq 2k-101}(P(\tau,\xi))\cdot f^\mathbf{e}||_{X_k}\leq
C||f^\mathbf{e}||_{Y^\mathbf{e}_k}.
\end{equation*}
It remains to write $\eta_{\leq2k-100}(P(\tau,\xi))\cdot
f^\mathbf{e}$ as in \eqref{es2*}. From \eqref{vc211}, we obtain
\begin{equation}\label{vc22}
\begin{split}
&\eta_{\leq2k-100}(P(\tau,\xi))\cdot f^\mathbf{e}(\xi_{1}^\mathbf{e}\mathbf{e}+\xi_{2}^\mathbf{e}\mathbf{e}^{\perp},\tau)\\
&=2^{k/2}\cdot \chi_{k,5}(\xi\cdot\overline{\mathbf{e}})\cdot
\frac{\eta_{\leq2k-100}(P(\tau,\xi))}{P(\tau,\xi)+i}\int_{\R}h''(y_1^\e,\xi_{2}^\mathbf{e},\tau)e^{-iy^\mathbf{e}_1\xi^\mathbf{e}_1}\,dy^\mathbf{e}_1.
\end{split}
\end{equation}
Let
\begin{eqnarray}\label{s}
S=\{(\xi,\tau): |\xi|\leq 2^{k+2}, |\xi\cdot\overline{\e}|\geq
2^{k-6}, |P(\tau,\xi)|\leq 2^{2k-80}\}.
\end{eqnarray}
It is easy to see that the right-hand side of \eqref{vc22} is
supported in $S$.

Now assume $(\xi,\tau)\in S$, and analyze the behavior of
$P(\tau,\xi)=\tau+\xi_1^2-\xi_2^2$. In view of \eqref{es22}, if we
fix $\tau$, $\xi_2^\e$, then there exists a
$t^*_\e=t^*_\e(\xi_2^\e,\tau)$, so that
\begin{eqnarray}\label{ret}
0&=&\tau+(t^*_\mathbf{e})^2(\cos^2\theta-\sin^2\theta)
-4t^*_\mathbf{e}\xi_{2}^\mathbf{e}\cos\theta\sin\theta\nonumber\\
&&-(\xi_{2}^\mathbf{e})^2(\cos^2\theta-\sin^2\theta).
\end{eqnarray}
Thus we get
\begin{eqnarray}\label{ret1}
P(\tau,\xi)&=&\tau+(\xi_{1}^\mathbf{e})^2(\cos^2\theta-\sin^2\theta)
-4\xi_{1}^\mathbf{e}\xi_{2}^\mathbf{e}\cos\theta\sin\theta\nonumber\\
&&-(\xi_{2}^\mathbf{e})^2(\cos^2\theta-\sin^2\theta)\nonumber\\
&=&(\xi_{1}^\mathbf{e}-t^*_\mathbf{e})[(\xi_{1}^\mathbf{e}+t^*_\mathbf{e})(\cos^2\theta-\sin^2\theta)
-4\xi_{2}^\mathbf{e}\cos\theta\sin\theta].
\end{eqnarray}
Denote
\begin{eqnarray}\label{k}
K(\tau,\xi)=(\xi_{1}^\mathbf{e}+t^*_\mathbf{e})(\cos^2\theta-\sin^2\theta)
-4\xi_{2}^\mathbf{e}\cos\theta\sin\theta,
\end{eqnarray}
we reduce
\begin{eqnarray}\label{rep}
 P(\tau,\xi)=(\xi_{1}^\mathbf{e}-t^*_\mathbf{e})K(\tau,\xi).
\end{eqnarray}

Now we show that $|K(\tau,\xi)|\geq 2^{k-10}$ for $(\xi,\tau)\in S$,
for proper $t^*_\mathbf{e}$. Notice that
\begin{eqnarray}\label{rep1}
\xi\cdot\overline{\mathbf{e}}=(\xi_{1}^\mathbf{e}\mathbf{e}+\xi_{2}^\mathbf{e}\mathbf{e}^{\perp})
\cdot\overline{\mathbf{e}}=\xi_{1}^\mathbf{e}(\cos^2\theta-\sin^2\theta)
-2\xi_{2}^\mathbf{e}\cos\theta\sin\theta
\end{eqnarray}
we get
\begin{eqnarray}\label{repkp}
K(\tau,\xi)=-(\xi_{1}^\mathbf{e}-t^*_\mathbf{e})(\cos^2\theta-\sin^2\theta)+2\xi\cdot\overline{\mathbf{e}}.
\end{eqnarray}
If $\cos^2\theta-\sin^2\theta=0$, then
$P(\tau,\xi)=\tau\pm2\xi^\e_1\xi_2^\e$, this case is easy. If
$\cos^2\theta-\sin^2\theta\neq0$, we can assume that at least one
solution satisfies
\begin{eqnarray}\label{kk}
|K(\tau,\xi)|=|(\xi_{1}^\mathbf{e}-t^*_\mathbf{e})(\cos^2\theta-\sin^2\theta)-2\xi\cdot\overline{\mathbf{e}}|>2^{k-20}\quad\text{
for }(\xi,\tau)\in S.
\end{eqnarray}
Otherwise, we denote $t_1,t_2$ to be the two solutions of
\eqref{ret}. Since \eqref{repkp} and $|\xi\cdot \overline{\e}|
\geq2^{k-6}$ in $S$, we have
\begin{eqnarray}\label{c11}
|(\xi_{1}^\mathbf{e}-t_i)(\cos^2\theta-\sin^2\theta)|\geq
2^{k-10}\quad\text{ for } i=1,2, \text{ and }(\xi,\tau)\in S,
\end{eqnarray}
and by the assumptation on $t_i$, we have
\begin{eqnarray}\label{c12}
P(\tau,\xi)=(\cos^2\theta-\sin^2\theta)(\xi_{1}^\mathbf{e}-t_1)(\xi_{1}^\mathbf{e}-t_2).
\end{eqnarray}
Combining \eqref{c11} and \eqref{c12}, we conclude that
$|P(\tau,\xi)|\geq2^{2k-20}$, which contradict with $(\xi,\tau)\in
S$. Then we select $t_\e^*$ to be the solution satisfies that
$|K(\tau,\xi)|>2^{k-20}$ for $(\xi,\tau)\in S$.
Furthermore, we have
\begin{eqnarray}\label{rex}
|\xi_{1}^\mathbf{e}-t^*_\mathbf{e}(\xi_2^\e,\tau)|\leq 2^{k-60}
\text{ for
}(\xi_{1}^\mathbf{e}\mathbf{e}+\xi_{2}^\mathbf{e}\mathbf{e}^{\perp},\tau)\in
S.
\end{eqnarray}

Now we show that $t^*_\mathbf{e}\in \R$. Let
$\mathfrak{Re}(t^*_\mathbf{e})$ and $\mathfrak{Im}(t^*_\mathbf{e})$
denote the real and imaginary part of $t^*_\mathbf{e}$, we notice
that
\begin{eqnarray}\label{m}
\mathfrak{Re}M^*_\e(\xi_{2}^\mathbf{e},\tau)=\mathfrak{Re}(t_\e^{*})(\cos^2\theta-\sin^2\theta)
-2\xi_{2}^\mathbf{e}\cos\theta\sin\theta.
\end{eqnarray}
 In view of
\eqref{rep1}, we have
\begin{eqnarray}\label{rem}
&&\xi\cdot\overline{\mathbf{e}}=(\xi_{1}^\mathbf{e}-\mathfrak{Re}(t_\e^{*}))
(\cos^2\theta-\sin^2\theta)+\mathfrak{Re}M_\mathbf{e}^{*}(\xi_{2}^\mathbf{e},\tau)\nonumber\\
&&\text{where }|\xi_{1}^\mathbf{e}-\mathfrak{Re}(t_\e^{*})|
\leq|\xi_{1}^\mathbf{e}-t^*_\mathbf{e}|\leq 2^{k-80} \text{ by
\eqref{rex}}.
\end{eqnarray}
From \eqref{rem}, we get that $|\mathfrak{Re}M_\mathbf{e}^{*}|\geq
2^{k-10}$ in $S$. In view of \eqref{ret}, we have
$$
0=2i\mathfrak{Re}(t_\e^{*})\mathfrak{Im}(t_\e^{*})(\cos^2\theta-\sin^2\theta)
-4i\mathfrak{Im}(t_\e^{*})\xi_2^\e\cos\theta\sin\theta=2i\mathfrak{Re}(M_\e^*)\mathfrak{Im}(t_\e^{*})
$$
which implies $\mathfrak{Im}(t_\e^{*})=0$, thus $t_\e^*\in \R$.

Then we can rewrite \eqref{rem} as
\begin{eqnarray}\label{rem1}
&&\xi\cdot\overline{\mathbf{e}}=(\xi_{1}^\mathbf{e}-t_\e^{*})
(\cos^2\theta-\sin^2\theta)+M_\mathbf{e}^{*}(\xi_{2}^\mathbf{e},\tau)\nonumber\\
&&\text{where }|\xi_{1}^\mathbf{e}-t^*_\mathbf{e}|\leq 2^{k-80}
\text{ by \eqref{rex}}.
\end{eqnarray}
In view of \eqref{k} and \eqref{mm}, we have
\begin{eqnarray}\label{rek}
&&K(\tau,\xi)=(\xi_{1}^\mathbf{e}-t_\e^{*})(\cos^2\theta-
\sin^2\theta)+2M^*_\e(\xi_{2}^\mathbf{e},\tau)\nonumber\\
&&\text{where }|\xi_{1}^\mathbf{e}-t_\e^{*}| \leq 2^{k-80} \text{ by
\eqref{rex}}.
\end{eqnarray}

Now we begin to prove the representation formula. In view of
\eqref{rep}, we have
\begin{eqnarray*}
&&\chi_{k,5}(\xi\cdot\overline{\mathbf{e}})\cdot
\frac{\eta_{\leq2k-100}(P(\tau,\xi))}{P(\tau,\xi)+i}\nonumber\\
&=&\chi_{k,5}(\xi\cdot\overline{\mathbf{e}})\cdot
\frac{\eta_{\leq2k-100}((\xi_{1}^\mathbf{e}-t^*_\mathbf{e})K(\tau,\xi))}{(\xi_{1}^\mathbf{e}-t^*_\mathbf{e})K(\tau,\xi)+i}.
\end{eqnarray*}
Combining \eqref{rek} and \eqref{rem1}, we get
\begin{equation}\label{vc23}
\begin{split}
&\chi_{k,5}(\xi\cdot\overline{\mathbf{e}})\cdot
\frac{\eta_{\leq2k-100}((\xi_{1}^\mathbf{e}-t^*_\mathbf{e})K(\tau,\xi))}{(\xi_{1}^\mathbf{e}-t^*_\mathbf{e})K(\tau,\xi)+i}\\
=&\chi_{k,5}(M_\mathbf{e}^{*})\cdot \frac{\eta_{\leq
k-100}(\xi_{1}^\mathbf{e}-t_\mathbf{e}^{*})}{(\xi_{1}^\mathbf{e}-t^*_\mathbf{e})
+i/2^k}\cdot\frac{1}{2M_\mathbf{e}^{*}}+E(\xi_1^\e,\xi_2^\e,\tau),
\end{split}
\end{equation}
where $E(\xi_1^\e,\xi_2^\e,\tau)$ is the error term with the
estimate
\begin{eqnarray}\label{e}
|E(\xi_1^\e,\xi_2^\e,\tau)|&\leq&
C\chi_{k,10}(M_\mathbf{e}^{*})\cdot\eta_{\leq
k-90}(\xi_{1}^\mathbf{e}-t_\e^{*})\nonumber\\&&\times
\Big[\frac{1}{2^{2k}}+(1+|P(\tau,\xi)|)^{-2}\Big].
\end{eqnarray}

We substitute \eqref{vc23} into \eqref{vc22} and notice that the
error term corresponding to $E(\xi_1^\e,\xi_2^\e,\tau)$ can be
controlled in $X_k$ (as  in Lemma \ref{lp3}). The  main term in the
right-hand side of \eqref{vc23} leads to the representation
\eqref{es2*}, with $h= h''$.

For \eqref{pp4}, differentiate the equation \eqref{ret} in $\tau$
axis, we have
\begin{eqnarray}\label{ran1}
1&=&2(\partial_{\tau}t_\mathbf{e}^*)[-t_\mathbf{e}^*(\cos^2\theta-\sin^2\theta)+2\xi_{2}^\mathbf{e}\cos\theta\sin\theta]\nonumber\\
&=&-2(\partial_{\tau}t_\mathbf{e}^*)M^*_\mathbf{e}
\end{eqnarray}
where $M^*_\e\sim 2^{k}$, thus we conclude \eqref{pp4}.
\end{proof}

The following local-smoothing estimates is the main lemma in this
paper, the corresponding lemma in Schrodinger case is Lemma 3.2 in
\cite{IK}.

\begin{lemma}[Local-smoothing estimate]\label{lp4}
If $k \in \mathbb{Z}_{+}$, $\mathbf{e}'\in \mathbb{S}^{1}$ and $f\in
Z_{k}$ then
\begin{eqnarray}\label{es2}
\|\mathcal{F}_{(2+1)}^{-1}[f\cdot \chi_{k,30}(\xi\cdot
\overline{\mathbf{e}'})]\|_{L^{\infty, 2}_{\mathbf{e}'}}\leq
C2^{-k/2}\|f\|_{Z_{k}}
\end{eqnarray}

\end{lemma}
\begin{proof}[\textbf{Proof of Lemma \ref{lp4}}]
In view of the space decomposition \eqref{lp11}. Assume first that
$f=g_j\in X_k$. In view of the definitions, it suffices to prove
that if $j\geq 0$ and $g_j$ is supported in $D_{k,j}$, then
\begin{equation}\label{gu41}
\Big|\Big|\int_{\mathbb{R}^{2+1}}e^{ix\cdot\xi}e^{it\tau}
g_j(\xi,\tau)\cdot\chi_{k,30}(\xi\cdot\overline{\mathbf{e}'}) \,d\xi
d\tau\Big|\Big|_{L^{\infty,2}_{\e'}}\leq
C2^{-k/2}2^{j/2}\|g_j\|_{L^2}.
\end{equation}
Let $g_j^\#(\xi,\tau)=g_j(\xi,\tau-\xi_1^2+\xi_2^2)$. By
H\"{o}lder's ine\-qua\-lity, for \eqref{gu41} it reduce to show
\begin{equation}\label{vy10}
\Big|\Big|\int_{\mathbb{R}^{2}}e^{ix\cdot\xi}e^{-it(\xi_1^2-\xi_2^2)}h(\xi)
\cdot\chi_{k,30}(\xi\cdot\overline{\mathbf{e}'})\,d\xi
\Big|\Big|_{L^{\infty,2}_{\e'}}\leq C2^{-k/2}\|h\|_{L^2}.
\end{equation}
which follows easily from Plancherel theorem and a change of
variables(let $\nu=\xi_1^2-\xi_2^2$, then
$\chi_{k,30}(\xi\cdot\overline{\mathbf{e}'})d\xi=2^{-k}\chi_{k,30}(\xi\cdot\overline{\mathbf{e}'})d\nu
d\xi_{2}^{\e'}$ in view of \eqref{lp22}).

Assume now that $f=f_{\mathbf{e}}\in Y_{k}^{\mathbf{e}}$,
 $k\geq 100$. We adopt the notations in Lemma \eqref{lp4*}.
The estimates in Lemma \ref{lp4*} show that(since
$|\chi_{k,30}(\xi\cdot\overline{\mathbf{e}'})-
\chi_{k,30}((t_\e^*\e+\xi_2^\e\e^\perp)\cdot\overline{\mathbf{e}'})|\leq
C2^{-k}|\xi_1^\e-t_\e^*|$),
\begin{equation*}
||f^{\mathbf{e}}\cdot
[\chi_{k,30}(\xi\cdot\overline{\mathbf{e}'})-\chi_{k,30}((t_\e^*\e+\xi_2^\e\e^\perp)\cdot\overline{\mathbf{e}'})]||_{X_k}\leq
C||f^{\mathbf{e}}||_{Y_k^{\mathbf{e}}}.
\end{equation*}
Since \eqref{gu41} was already proved for $f\in X_k$, it suffices to
show that
\begin{equation}\label{gu47}
\begin{split}
\Big|\Big|&\int_{\R^2\times\mathbb{R}}e^{ix_1^\e\xi_1^\e}e^{ix_2^\e\xi_2^\e}e^{it\tau}
f^{\mathbf{e}}(\xi_1^\e\e+\xi_2^\e\e^\perp,\tau)\chi_{k,5}(\xi\cdot\overline{\e})\\
&\times\chi_{k,30}((t_\e^*\e+\xi_2^\e\e^\perp)\cdot\overline{\e'})\,d\xi_1^\e
d\xi_2^\e d\tau\Big|\Big|_{L^{\infty,2}_{\e'}}\leq
C2^{-k/2}\|f^{\mathbf{e}}\|_{Y_k^{\mathbf{e}}}.
\end{split}
\end{equation}
By substituding \eqref{es2*} to \eqref{gu47} and then integrating
the left-hand side of \eqref{gu47} according to the variable
$\xi_1^\e$, we can reduce \eqref{gu47} to show
\begin{equation}\label{gu48}
\begin{split}
\Big|\Big|\int_{\R\times\mathbb{R}}&e^{ix_1^\e
t_\e^*}e^{ix_2^\e\xi_2^\e}e^{it\tau}
\cdot 2^{-k/2}h'(\xi_2^\e,\tau)\chi_{k,10}((t_\e^*\e+\xi_2^\e\e^\perp)\cdot\overline{\e})\\
&\times\chi_{k,30}((t_\e^*\e+\xi_2^\e\e^\perp)\cdot\overline{\e'})\,d\xi_2^\e
d\tau\Big|\Big|_{L^{\infty,2}_{\e'}}\leq
C2^{-k/2}\|h'\|_{L^2_{\xi_2^\e,\tau}},
\end{split}
\end{equation}
for any $h'\in L^2(\R\times\R)$. As before,
$\e=(\cos\theta,\sin\theta)$, we use the substitutions
$\tau=-(\mu\cos\theta+\xi_{2}^\e\sin\theta)^{2}+(\mu\sin\theta-\xi_{2}^\e\cos\theta)^{2}\triangleq
Q(\mu,\xi_2^\e)$ (so $t_\e^*=\mu$), and
$$h''(\xi,\mu)=2^{-k/2}\cdot\partial_{\mu}(Q(\mu,\xi_2^\e))\cdot
\chi_{k,5}((\mu\e+\xi_2^\e\e^\perp)\cdot\overline{\e})\cdot
h'(\xi_2^\e,Q(\mu,\xi_2^\e)).$$ We notice that
$\partial_{\mu}(Q(\mu,\xi_2^\e))=
(\mu\e+\xi_2^\e\e^\perp)\cdot\overline{\mathbf{e}}\approx 2^k$, so
$||h''||_{L^2}\leq C||h'||_{L^2}$. For \eqref{gu48} it suffices to
prove that
\begin{equation*}
\begin{split}
\Big|\Big|\int_{\R\times\mathbb{R}}&e^{ix_1^\e\mu}e^{ix_2^\e\xi_2^\e}e^{-itQ(\mu,\xi_2^\e)}\cdot
 h''(\xi_2^\e,\mu)\\
&\times\chi_{k,30}((\mu\e+\xi_2^\e\e^\perp)\cdot\overline{\e'})\,d\xi_2^\e
d\mu\Big|\Big|_{L^{\infty,2}_{\e'}}\leq
C2^{-k/2}\|h''\|_{L^2_{\xi_2^\e,\mu}},
\end{split}
\end{equation*}
which follows from \eqref{vy10}.
\end{proof}

\begin{remark}\label{re1}
If we remove the cut-off function $\chi_{k,30}(\xi\cdot
\overline{\mathbf{e}'})$, the estimate
\begin{eqnarray}\label{es22}
\|\mathcal{F}_{(2+1)}^{-1}(f)\|_{L^{\infty, 2}_{\mathbf{e}'}}\leq
C2^{-k/2}\|f\|_{Z_{k}}
\end{eqnarray}
is not ture. Furthermore, the following inequality is false
\begin{eqnarray}\label{1}
\|\mathcal{F}^{-1}_{(2+1)}(f)\|_{L^{\infty,2}_{\e}}\leq
C(k+1)\|f\|_{Z_{k}}
\end{eqnarray}
To see this, let $\e=(\cos (\pi/4),\sin(\pi/4))$, then
$P(\tau,\xi)=\tau+\xi_{1}^{\e}\xi_{2}^{\e}$, define
$f(\tau,\xi_{1}^{\e}\e+\xi_{2}^{\e}\e^{\perp})=\mathbf{1}_{Q}(\tau,\xi_{1}^{\e}\e+\xi_{2}^{\e}\e^{\perp})$,
where
$$
Q=\{N\leq|\xi_{1}^{\e}|\leq 2N; |\xi_{2}^{\e}|\leq 1/(4N);
|\tau|\leq 1/2\}.
$$
\end{remark}

Next we give a maximal function estimate.
\begin{lemma}[Maximal Function Estimate]\label{lp5}
If $k \in \mathbb{Z}_{+}$, $\mathbf{e}\in \mathbb{S}^{1}$ and $f\in
Z_{k}$ then
\begin{eqnarray}\label{es5}
\|1_{[-2,2]}(t)\cdot\mathcal{F}_{(2+1)}^{-1}(f)\|_{L^{2, \infty
}_{\mathbf{e}}}\leq C2^{k/2}(k+1)^{2}\|f\|_{Z_{k}}
\end{eqnarray}
\end{lemma}
\begin{proof}[\textbf{Proof of Lemma \ref{lp5}}]
In view of Lemma \ref{lp3},  it suffices to prove that
\begin{eqnarray}\label{es51}
\|1_{[-2,2]}(t)\cdot\mathcal{F}_{(2+1)}^{-1}(g_{j})\|_{L^{2, \infty
}_{\mathbf{e}}}\leq C2^{k/2}(k+1)2^{j/2}\|g_{j}\|_{L^{2}},
\end{eqnarray}
for any function $g_j$ supported in $D_{k,j}$. Denote
$g_{j}^{\#}(\xi,\tau)=g_{j}(\xi,\tau-\xi_{1}^{2}+\xi_{2}^{2})$, the
left-hand side of \eqref{es51} is dominated by
\begin{eqnarray*}
\int_{[-2^{j+1},2^{j+1}]}\Big\|1_{[-2,2]}(t)\cdot\int_{\mathbb{R}^{2}}
g^{\#}_{j}(\xi,\mu)e^{ix\cdot\xi}e^{-it(\xi_{1}^{2}-\xi_{2}^{2})}d\xi\Big\|_{L^{2,
\infty }_{\mathbf{e}}}d\mu.
\end{eqnarray*}
Thus for \eqref{es51} it suffices to prove that
\begin{eqnarray}\label{es52}
\Big\|1_{[-2,2]}(t)\cdot\int_{\mathbb{R}^{2}}
h(\xi)e^{ix\cdot\xi}e^{-it(\xi_{1}^{2}-\xi_{2}^{2})}d\xi\Big\|_{L^{2,
\infty }_{\mathbf{e}}}\leq C2^{k/2}(k+1)\|h\|_{L_{\xi}^{2}}
\end{eqnarray}
for any function $h$ supported in the set $\{\xi\in\mathbb{R}^{2}:
|\xi|\leq 2^{k+1}\}$.

Let $\mathbf{e}= (\cos\theta, \sin\theta)$, rotate the x-axes to the
$\mathbf{e}$ direction, then \eqref{es52} changes to
\begin{eqnarray}\label{ple2}
\Big\|1_{[-2,2]}(t)\int_{\mathbb{R}^{2}}e^{
ix_{1}\xi_{1}}e^{ix_{2}\xi_{2} }e^{-itQ(\xi,\e)}\eta_{\leq
k+2}(\xi)h(\xi)d\xi\Big\|_{L^{2}_{x_{1}}L^{\infty}_{x_{2},t}}\lesssim
2^{k/2}(k+1)\|h\|_{L^{2}}
\end{eqnarray}
where $Q(\xi,\e)=
(\xi_{1}\cos\theta+\xi_{2}\sin\theta)^{2}-(\xi_{1}\sin\theta-\xi_{2}\cos\theta)^{2}$.
We use standard $TT^{*}$ argument to prove \eqref{ple2}, it means to
show that
\begin{eqnarray}\label{T}
\Big\|1_{[-4,4]}(t)\int_{\mathbb{R}^{2}}e^{
ix_{1}\xi_{1}}e^{ix_{2}\xi_{2} }e^{-itQ(\xi,\e)}\eta_{\leq
k+2}(\xi)d\xi\Big\|_{L^{1}_{x_{1}}L^{\infty}_{x_{2},t}}\lesssim
2^{k} (k+1)^{2}.
\end{eqnarray}
Notice that
\begin{eqnarray*}\label{ple}\Big|1_{[-4,4]}(t)\int_{\mathbb{R}^{2}}e^{
ix_{1}\xi_{1}}e^{ix_{2}\xi_{2} }e^{-itQ(\xi,\e)}\eta_{\leq
k+2}(\xi)d\xi\Big|\lesssim 2^{2k},
\end{eqnarray*}
rotate again, we get
\begin{eqnarray*}\label{ple}&&\Big|1_{[-4,4]}(t)\int_{\mathbb{R}^{2}}e^{
ix_{1}\xi_{1}}e^{ix_{2}\xi_{2} }e^{-itQ(\xi,\e)}\eta_{\leq
k+2}(\xi)d\xi\Big|\\
&=&\Big|1_{[-4,4]}(t)\int_{\mathbb{R}^{2}}e^{
ix_{1}\xi_{1}}e^{ix_{2}\xi_{2}
}e^{-it(\xi_{1}^{2}-\xi_{2}^{2})}\eta_{\leq k+2}(\xi)d\xi\Big|
\lesssim |t|^{-1}.
\end{eqnarray*}
Integration by parts when $|x_{1}|>2^{k+10}|t|$, then
\begin{eqnarray*}\label{ple3}
&&\Big|1_{[-4,4]}(t)\int_{\mathbb{R}^{2}}e^{
ix_{1}\xi_{1}}e^{ix_{2}\xi_{2} }e^{-itQ(\xi,\e)}\eta_{\leq
k+2}(\xi)d\xi\Big|\lesssim \frac{2^{2k}}{(1+2^{k}|x_{1}|)^{2}}.
\end{eqnarray*}
We collect all the estimates above and let $K(x_{1},x_{2},t)$ denote
the function in the left-hand side of \eqref{T}, then
\begin{eqnarray*}\label{ple4}
\sup_{x_{2},|t|<4}|K(x_{1},x_{2},t)|\lesssim
2^{k}|x_{1}|^{-1}\mathbf{1}_{\{2^{-k}\leq|x_1|<2^{k}\}}+\frac{2^{2k}}{(1+2^{k}|x_{1}|)^{2}}.
\end{eqnarray*}
The bound \eqref{T} follows.
\end{proof}

\begin{lemma}\label{lp6}
If $k \in \mathbb{Z}_{+}$, $t\in \mathbb{R}$ and $f\in Z_{k}$, then
\begin{eqnarray}\label{es4}
\sup_{t\in \mathbb{R}}\|\mathcal{F}_{(2+1)}^{-1}(f)(\cdot,
t)\|_{L^{2}_{x}}\leq C\|f\|_{Z_{k}}.
\end{eqnarray}
Thus
\begin{eqnarray}\label{es41}
\|\mathcal{F}_{(2+1)}^{-1}(f)(\cdot, t)\|_{L^{\infty}_{t,x}}\leq
C2^{k}\|f\|_{Z_{k}}.
\end{eqnarray}
\end{lemma}
\begin{proof}[\textbf{Proof of Lemma \ref{lp6}}] By Plancherel
theorem, it suffices to prove that
\begin{eqnarray}\label{es42}
\Big\|\int_{\mathbb{R}}f(\xi,\tau)e^{it\tau}d\tau\Big\|_{L^{2}_{\xi}}\leq
C\|f\|_{Z_{k}}.
\end{eqnarray}
By Lemma \ref{lp1}, we assume that $f=g_{j}\in X_{k}$.
\begin{eqnarray}\label{es43}
\Big\|\int_{\mathbb{R}}g_{j}(\xi,\tau)e^{it\tau}d\tau\Big\|_{L^{2}_{\xi}}\leq
C\|g_{j}(\xi,\tau)\|_{L^{2}_{\xi}L^{1}_{\tau}}\leq
C2^{j/2}\|g_{j}\|_{L^{2}_{\xi,\tau}}
\end{eqnarray}
which gives \eqref{es42} in this case.

Turn to the case $k\geq 100$ and $f=f_{\mathbf{e}}\in
Y_{k}^\mathbf{e}$, $\mathbf{e}\in
\{\mathbf{e}_{1},\ldots,\mathbf{e}_{L}\}$, we need to prove that
\begin{eqnarray}\label{es44}
\Big\|\int_{\mathbb{R}}f_{\mathbf{e}}(\xi,\tau)e^{it\tau}d\tau\Big\|_{L^{2}_{\xi}}\leq
C\|f_{\mathbf{e}}\|_{Y_{k}^\mathbf{e}}.
\end{eqnarray}
By writing
\begin{eqnarray*}
h_{\mathbf{e}}(x,t)=2^{-k/2}\mathcal{F}^{-1}_{(2+1)}[(P(\tau,\xi)+i)\cdot
f_{\mathbf{e}}](x,t)
\end{eqnarray*}
we get
\begin{eqnarray*}
f_{\mathbf{e}}(\xi,\tau)=\chi_{k,10}(\xi\cdot
\overline{\mathbf{e}})\cdot
\frac{2^{k/2}}{P(\tau,\xi)+i}\mathcal{F}_{(2+1)}(h_{\mathbf{e}})(\xi,\tau)
\end{eqnarray*}
Let $\xi=\xi^\mathbf{e}_{1}\e+\xi^\mathbf{e}_{2}\mathbf{e}^{\perp}$,
$x=x^\mathbf{e}_{1}\e+x^\mathbf{e}_{2}\mathbf{e}^{\perp}$. For
\eqref{es44} it suffices to prove that
\begin{eqnarray}\label{es45}
2^{k/2}\Big\|\chi_{k,10}(\xi\cdot
\overline{\mathbf{e}})\int_{\mathbb{R}}\frac{1}{P(\xi,\tau)+i}\mathcal{F}_{(2+1)}(h)(\xi^\mathbf{e}_{1}\e+
\xi^\mathbf{e}_{2}\mathbf{e}^{\perp},\tau)e^{it\tau}d\tau\Big\|_{L^{2}_{\xi}}\leq
C\|h\|_{L_{\mathbf{e}}^{1,2}}
\end{eqnarray}
for any $h\in \mathcal{S}(\mathbb{R}^{d}\times\mathbb{R})$ and $t\in
\mathbb{R}$. As in the proof of Lemma \ref{lp4}, we define
\begin{eqnarray*}
h'(x^\mathbf{e}_{1},\xi_{2}^\mathbf{e},\tau)=\int_{\mathbb{R}\times
\mathbb{R}}h(x^\mathbf{e}_{1}\e+x^\mathbf{e}_{2}\mathbf{e}^{\perp},t)e^{-i(x^\mathbf{e}_{2}\cdot
\xi^\mathbf{e}_{2}+t\tau)}dx^\mathbf{e}_{2}dt
\end{eqnarray*}
So
\begin{eqnarray*}
\mathcal{F}_{(2+1)}(h)(\xi^\mathbf{e}_{1}\e+
\xi^\mathbf{e}_{2}\mathbf{e}^{\perp},\tau)=\int_{\mathbb{R}}h'(x^\mathbf{e}_{1},\xi_{2}^\mathbf{e},\tau)e^{-ix^\mathbf{e}_{1}\cdot
\xi^\mathbf{e}_{1}}dx^\mathbf{e}_{1},
\end{eqnarray*}
and
$\|h\|_{L^{1,2}_{\mathbf{e}}}=C\|h'\|_{L_{x^\mathbf{e}_{1}}^{1}L_{\xi^\mathbf{e}_{2},\tau}^{2}}$.
Let
\begin{eqnarray*}
h_{t}^{*}(x^\mathbf{e}_{1},\xi_{2}^\mathbf{e},\mu)=\int_{\mathbb{R}}\frac{1}{\tau
+\mu+i}h'(x^\mathbf{e}_{1},\xi_{2}^\mathbf{e},\tau)e^{it\tau}d\tau.
\end{eqnarray*}
The boundedness of Hilbert transform on $L^{2}(\mathbb{R})$ gives
\begin{eqnarray*}
\|h_{t}^{*}(x^\mathbf{e}_{1},\xi_{2}^\mathbf{e},\mu)\|_{L^{2}_{\xi_{2}^\mathbf{e},\mu}}\leq
C\|h'(x^\mathbf{e}_{1},\xi_{2}^\mathbf{e},\tau)\|_{L^{2}_{\xi_{2}^\mathbf{e},\tau}}
\text{ for any }x^\mathbf{e}_{1},t\in \mathbb{R}.
\end{eqnarray*}
Thus for \eqref{es45}, it suffices to prove that
\begin{eqnarray}\label{es46}
2^{k/2}\Big\|\chi_{k,10}(\xi\cdot
\overline{\mathbf{e}})\int_{\mathbb{R}}h_{t}^{*}(x^\mathbf{e}_{1},\xi_{2}^\mathbf{e},\xi_{1}^{2}-
\xi_{2}^{2})dx_{1}^\mathbf{e}\Big\|_{L^{2}_{\xi}}\leq
C\|h_{t}^{*}\|_{L_{x^\mathbf{e}_{1}}^{1}L_{\xi^\mathbf{e}_{2},\tau}^{2}}.
\end{eqnarray}
Just notice $
\partial_{\xi_{1}^\mathbf{e}}(\xi_{1}^{2}-
\xi_{2}^{2})=2\xi\cdot \overline{\mathbf{e}} $, \eqref{es46} follows
from changes of variables.
\end{proof}

\section{Linear Estimates}

In this section, we prove two linear estimates for the smi-group
$e^{it\square_x}$ by following some ideas in \cite{IK}. For
$\sigma\geq 0$ we define the normed spaces
\begin{equation}\label{no5}
F^\sigma=\{u\in
C(\mathbb{R}:H^\infty):\|u\|_{F^\sigma}^2=\sum_{k=0}^\infty
2^{2\sigma k}\|\eta_k^{(d)}(\xi)\cdot
\mathcal{F}_{(d+1)}u\|_{Z_k}^2<\infty\},
\end{equation}
and
\begin{equation}\label{no6}
\begin{split}
N^\sigma=&\{u\in C(\mathbb{R}:H^\infty):\\
&\|u\|_{N^\sigma}^2=\sum_{k=0}^\infty 2^{2\sigma
k}\|\eta_k^{(d)}(\xi)\cdot (\tau+\xi_1^2-\xi_2^2+i)^{-1}\cdot
\mathcal{F}_{(d+1)}u\|_{Z_k}^2<\infty\}.
\end{split}
\end{equation}

For $\phi\in H^\infty$ let $W(t)\phi\in C(\mathbb{R}:H^\infty)$
denote the solution of the free Schr\"{o}dinger evolution
\begin{equation}\label{ni1}
[W(t)\phi](x,t)=c_{0}\int_{\mathbb{R}^{2}}e^{ix\cdot\xi}
e^{-it(\xi_{1}^{2}-\xi_{2}^{2})}\mathcal{F}_{(2)}(\phi)(\xi)d\xi.
\end{equation}
Assume $\psi:\mathbb{R}\to[0,1]$ is an even smooth function
supported in the interval $[-8/5,8/5]$ and equal to $1$ in the
interval $[-5/4,5/4]$.

\begin{lemma}\label{le1}
If $\sigma\geq 0$ and $\phi\in H^{\infty}$ then $\psi(t)\cdot
[W(t)\phi]\in F^\sigma$ and
\begin{equation*}
\|\psi(t)\cdot [W(t)\phi]\|_{F^{\sigma}}\leq
C_\sigma\|\phi\|_{H^\sigma}.
\end{equation*}
\end{lemma}
\begin{proof}[Proof of Lemma \ref{le1}] A straightforward computation shows that
\begin{equation*}
\mathcal{F}[\psi(t)\cdot [W(t)\phi]](\xi,\tau)=
\mathcal{F}_{(2)}(\phi)(\xi)\cdot
\mathcal{F}_{(1)}(\psi)(P(\tau,\xi)).
\end{equation*}
Then, directly from the definitions,
\begin{equation*}
\begin{split}
\|\psi(t)\cdot
[W(t)\phi]\|^2_{F^{\sigma}}&=\sum_{k\in\mathbb{Z}_+}2^{2\sigma k}
\|\eta_k^{(2)}(\xi)\mathcal{F}_{(2)}(\phi)(\xi)\mathcal{F}_{(1)}(\psi)(P(\tau,\xi))\|^2_{Z_k}\\
&\leq \sum_{k\in\mathbb{Z}_+}2^{2\sigma
k}\|\eta_k^{(2)}(\xi)\mathcal{F}_{(2)}
(\phi)(\xi)\mathcal{F}_{(1)}(\psi)(P(\tau,\xi))\|^2_{X_k}\\
&\leq C\sum_{k\in\mathbb{Z}_+}2^{2\sigma k}\|\eta_k^{(2)}(\xi)\cdot \mathcal{F}_{(2)}(\phi)(\xi)\|^2_{L^2}\\
&\leq C\|\phi\|^2_{H^\sigma}.
\end{split}
\end{equation*}
\end{proof}

\begin{lemma}\label{le2}
If $\sigma\geq 0$ and $u\in N^{\sigma}$ then $\psi(t)\cdot
\int_0^tW(t-s)(u(s))\,ds\in F^\sigma$ and
\begin{equation*}
\Big|\Big|\psi(t)\cdot
\int_0^tW(t-s)(u(s))\,ds\Big|\Big|_{F^{\sigma}}\leq
C||u||_{N^{\sigma}}.
\end{equation*}
\end{lemma}
\begin{proof}[Proof of Lemma \ref{le2}] A straightforward computation shows that
\begin{equation*}
\begin{split}
\mathcal{F}\Big[\psi(t)\cdot& \int_0^tW(t-s)(u(s))ds\Big](\xi,\tau)\\
&=c\int_\mathbb{R}\mathcal{F}(u)(\xi,\tau')\frac{\widehat{\psi}(\tau-\tau')-\widehat{\psi}
(P(\tau,\xi))}{P(\tau',\xi)}d\tau',
\end{split}
\end{equation*}
where, for simplicity of notation,
$\widehat{\psi}=\mathcal{F}_{(1)}(\psi)$. For $k\in\mathbb{Z}$ let
$$f_k(\xi,\tau')=\mathcal{F}(u)(\xi,\tau')\cdot \eta_k^{(2)}(\xi)\cdot (P(\tau',\xi)+i)^{-1}.$$
For $f\in Z_k$ let
\begin{equation}\label{ar202}
T(f)(\xi,\tau)=\int_\mathbb{R}f(\xi,\tau')\frac{\widehat{\psi}(\tau-\tau')-
\widehat{\psi}(P(\tau,\xi))}{P(\tau',\xi)}(P(\tau',\xi)+i)\,d\tau'.
\end{equation}
where $P(\tau,\xi)=\tau+\xi_1^2-\xi_2^2$. In view of the
definitions, it suffices to prove that
\begin{equation}\label{ni5}
||T||_{Z_k\to Z_k}\leq C\text{ uniformly in }k\in\mathbb{Z}.
\end{equation}

To prove \eqref{ni5} we use the representation \eqref{lp11}. Assume
first that $f=g_j$ is supported in $D_{k,j}$. Let
$g_j^\#(\xi,\mu')=g_j(\xi,\mu'-\xi_1^2+\xi_2^2)$ and
$[T(g)]^\#(\xi,\mu)=T(g)(\xi,\mu-\xi_1^2+\xi_2^2)$. Then,
\begin{equation}\label{ni6}
[T(g)]^\#(\xi,\mu)=\int_\mathbb{R}g_j^\#(\xi,\mu')\frac{\widehat{\psi}(\mu-\mu')-
\widehat{\psi}(\mu)}{\mu'}(\mu'+i)\,d\mu'.
\end{equation}
We use the elementary bound
\begin{equation*}
\Big|\frac{\widehat{\psi}(\mu-\mu')-\widehat{\psi}(\mu)}{\mu'}(\mu'+i)\Big|\leq
C[(1+|\mu|)^{-4}+(1+|\mu-\mu'|)^{-4}].
\end{equation*}
Then, using \eqref{ni6},
\begin{equation*}
\begin{split}
|T(g)^\#(\xi,\mu)|&\leq C(1+|\mu|)^{-4}\cdot 2^{j/2}\Big[\int_{\mathbb{R}}|g_j^\#(\xi,\mu')|^2\,d\mu'\Big]^{1/2}\\
&+C\mathbf{1}_{[-2^{j+10},2^{j+10}]}(\mu)\int_{\mathbb{R}}|g_j^\#(\xi,\mu')|(1+|\mu-\mu'|)^{-4}\,d\mu'.
\end{split}
\end{equation*}
It follows from the definition of the spaces $X_k$ that
\begin{equation}\label{ni7}
||T||_{X_k\to X_k}\leq C\text{ uniformly in }k\in\mathbb{Z}_+,
\end{equation}
as desired.

Assume now that $f=f^{\mathbf{e}}\in Y^\mathbf{e}_k$,  $k\geq 100$,
$\mathbf{e}\in\{\mathbf{e}_1,\ldots,\mathbf{e}_L\}$. We write
\begin{equation*}
f^{\mathbf{e}}(\xi,\tau')=\frac{\tau'+\xi_1^2-\xi_2^2}{\tau'+\xi_1^2-\xi_2^2+i}f^{\mathbf{e}}
(\xi,\tau')+\frac{i}{\tau'+\xi_1^2-\xi_2^2+i}f^{\mathbf{e}}(\xi,\tau').
\end{equation*}
Using Lemma \ref{lp3},
$||i(\tau'+\xi_1^2-\xi_2^2+i)^{-1}f^{\mathbf{e}}(\xi,\tau')
||_{X_k}\leq C||f^\e||_{Y_k^\mathbf{e}}$. In view of \eqref{ar202}
and \eqref{ni7}, for \eqref{ni5} it suffices to prove that
\begin{equation}\label{ni8}
\Big|\Big|\int_\mathbb{R}f^{\mathbf{e}}(\xi,\tau')\widehat{\psi}(\tau-\tau')\,d\tau'\Big|\Big|_{Z_k}+
\Big|\Big|\widehat{\psi}(P(\tau,\xi))\int_\mathbb{R}f^{\mathbf{e}}(\xi,\tau')\,d\tau'\Big|\Big|_{X_k}\leq
C||f^{\mathbf{e}}||_{Y_k^{\mathbf{e}}}.
\end{equation}
The bound for the first term in the left-hand side of \eqref{ni8}
follows easily from the definition. The bound for the second term in
the left-hand side of \eqref{ni8} follows from \eqref{es44} with
$t=0$.
\end{proof}

\section{Trilinear estimates}
In this section, we set up a trilinear estimate. First we
 reduce the nonlinear term of \eqref{t21}. Let
$u(t)\in C(\mathbb{R}: H^{\infty})$, and write the nonlinear term of
\eqref{t21} as
\begin{eqnarray}\label{tri}
\left \{\begin{array}{ll} \displaystyle \mathfrak{F}(u)
=\frac{2\bar{u}}{1+u\bar{u}}[(\partial_{x_1}u)^{2}-(\partial_{x_2}u)^{2}]+ib(\phi_{x_1}u_{x_2}+\phi_{x_2}u_{x_1})
\\
\Delta
\phi=4i\Big[\Big(\frac{u_{x_1}\bar{u}}{1+|u|^2}\Big)_{x_2}-\Big(\frac{u_{x_2}\bar{u}}{1+|u|^2}\Big)_{x_1}\Big],
\end{array}
\right.
\end{eqnarray}
thus we have
\begin{eqnarray}\label{phi}
\phi_{x_1}u_{x_2}+\phi_{x_2}u_{x_1} &=&4i\Big[\frac{\partial_{x_1}
\partial_{x_2}}{\Delta}\Big(\frac{u_{x_1}\bar{u}}{1+|u|^2}\Big)-\frac{\partial_{x_1}
^2}{\Delta}\Big(\frac{u_{x_2}\bar{u}}{1+|u|^2}\Big)\Big]u_{x_2}\nonumber\\
&&+4i\Big[\frac{
\partial^2_{x_2}}{\Delta}\Big(\frac{u_{x_1}\bar{u}}{1+|u|^2}\Big)-\frac{\partial_{x_1}
\partial_{x_2}}{\Delta}\Big(\frac{u_{x_2}\bar{u}}{1+|u|^2}\Big)\Big]u_{x_1}.
\end{eqnarray}
Here $\frac{\partial_{x_1}
\partial_{x_2}}{\Delta}$, $\frac{\partial_{x_1}
^2}{\Delta}$, $\frac{
\partial^2_{x_2}}{\Delta}$ are defined by Fourier multipliers
$\frac{\xi_1 \xi_2}{\xi^2_1 +\xi^2_2}$, $\frac{\xi_1 ^2}{\xi^2_1
+\xi^2_2}$, $\frac{ \xi_2^2}{\xi^2_1 +\xi^2_2}$. And we use
$\mathcal{R}_i (i=1,2,3)$ to denote these three $L^2$-bounded
operators accordingly. Thus the nonlinear term of \eqref{t21} can be
written as
\begin{eqnarray}\label{F}
\mathfrak{F}(u) &=&\mathcal{N}_{0}(u)[(\partial_{x_{1}}u)^{2}-
(\partial_{x_{2}}u)^{2}]-4b\mathcal{R}_1(\mathcal{N}_{0}(u)\partial_{x_1}
u)\partial_{x_2} u\nonumber\\
&&+4b\mathcal{R}_2(\mathcal{N}_{0}(u)\partial_{x_2} u)\partial_{x_2}
u-4b\mathcal{R}_3(\mathcal{N}_{0}(u)\partial_{x_1} u)\partial_{x_1}
u\\
&&+4b\mathcal{R}_1(\mathcal{N}_{0}(u)\partial_{x_2} u)\partial_{x_1}
u\nonumber,
\end{eqnarray}
where $\mathcal{N}_{0}(u)=2\bar{u}/(1+|u|^2)$.

We consider here the nonlinear term
\begin{eqnarray}\label{N}
\mathcal{N}(u)=\psi(t)\mathfrak{F}(u)\in C(\mathbb{R}: H^{\infty}),
\end{eqnarray}
and are looking for the control of
$$
\|\mathcal{N}(u)-\mathcal{N}(v)\|_{N^{\sigma}}, \quad \sigma>3/2,
$$
where $u,v\in F^{\sigma}$.

For $k\in \mathbb{Z_{+}}$ we define the normed spaces
$$
\widetilde{Z}_{k}=\{f\in L^{2}(\mathbb{R}^{2}\times\mathbb{R}): supp
f \in I_{k}\times\mathbb{R},\quad  \|f\|_{\widetilde{Z}_k}< \infty\}
$$
where
\begin{eqnarray}\label{Z}
\|f\|_{\widetilde{Z}_k}&=&2^{-k/2}(k+1)^{-2}\sup_{\mathbf{e}\in
S^{1}}\|1_{[-2,2]}(t)\mathcal{F}^{-1}_{(2+1)}(f)\|_{L^{2,\infty}_{\mathbf{e}}}\nonumber\\&&+2^{k/2}\sup_{\mathbf{e}\in
S^{1}}\|\mathcal{F}^{-1}_{(2+1)}[f\cdot\chi_{k,20}(\xi\cdot
\overline{\mathbf{e}})]\|_{L^{\infty,2}_{\mathbf{e}}}.
\end{eqnarray}
For $\sigma\geq 0$ we define the normed spaces
$$
\widetilde{F}^{\sigma}=\left\{u\in C(\mathbb{R}:H^{\infty}): \quad
\|u\|^{2}_{\widetilde{F}^{\sigma}}=\sum_{k=0}^{\infty}2^{2\sigma
k}\|\eta_{k}(\xi)\mathcal{F}^{-1}_{(2+1)}(u)\|^{2}_{\widetilde{Z}_k}<
\infty\right\},
$$
and
$$
\overline{F}^{\sigma}=\left\{u\in C(\mathbb{R}:H^{\infty}): \quad
\overline{u}\in F^{\sigma},\quad
\|u\|_{\overline{F}^{\sigma}}=\|\overline{u}\|_{F^{\sigma}}\right\}.
$$
It is easy to see that
$$
\|\overline{u}\|_{F^{\sigma}+\overline{F}^{\sigma}}=\|u\|_{F^{\sigma}+\overline{F}^{\sigma}}
$$
By Lemma \ref{lp4} and Lemma \ref{lp5}, we obtain
\begin{eqnarray}\label{em}
\left \{\begin{array}{ll} \displaystyle \|f\|_{\widetilde{Z}_k}\leq
C\|f\|_{Z_{k}}\text{ for any } k\in \mathbb{Z}_{+} \text{ and }f\in Z_{k}\\
\|u\|_{\widetilde{F}^{\sigma}}\leq
C\|u\|_{F^{\sigma}+\overline{F}^{\sigma}}\text{ for any } \sigma\geq
0 \text{ and }u\in F^{\sigma}+\overline{F}^{\sigma}.
\end{array}
\right.
\end{eqnarray}

For $\sigma\in \mathbb{R}$ let $J^{\sigma}$ denote the operator
defined by the Fourier multiplier $(\xi,\tau)\rightarrow
(1+|\xi|^{2})^{\sigma/2}$.

\begin{lemma}\label{n1}
For $\sigma>3/2$ we have
\begin{eqnarray}\label{n1e0}
\|\mathcal{R}(J^{1}(u_1)J^{1}(u_2))J^{1}(u_3)\|_{N^{\sigma}}\leq
C_{\sigma}\|u_{1}\|_{\widetilde{F}^{\sigma}}\|u_{2}\|_
{\widetilde{F}^{\sigma}}\|u_{3}\|_{\widetilde{F}^{\sigma}},
\end{eqnarray}
where $\mathcal{R}$ denotes $ I$,$\frac{\partial_x
\partial_y}{\Delta}$, $\frac{\partial_x
^2}{\Delta}$, or $\frac{
\partial^2_y}{\Delta}$; and $J^{\sigma}$
denotes the operator defined by the Fourier multiplier
$(\xi,\tau)\rightarrow (1+|\xi|^{2})^{\sigma/2}$.
\end{lemma}
\begin{proof}[\textbf{Proof of Lemma \ref{n1}}]Let $k_{\max}=\max\{k_1,k_2,k_3\}$, and similarly $k_{\med}$
and $k_{\min}$. In view of the definition, for \eqref{n1e0}, it
suffices to prove a dyadic trilinear estimate
\begin{eqnarray}\label{n1e1}
&&2^{k_{1}+k_{2}+k_{3}}\|(P(\tau,\xi)+i)^{-1}\mathcal{F}_{(2+1)}[P_{k}(\mathcal{R}(P_{k_1}u_1P_{k_2}
u_2)P_{k_3}u_3)]\|_{Z_{k}}\nonumber\\
&\leq&
C2^{\frac{k_{\max}-k}{2}}2^{\frac{3k_{\med}+3k_{\min}}{2}}(k_{\med}+1)^3(k_{\min}+1)^3\nonumber\\
&&\times\|P_{k_1}u_1\|_{\widetilde{Z}_{k_1}}\|P_{k_2}u_2\|_{\widetilde{Z}_{k_2}}\|P_{k_3}u_3\|_{\widetilde{Z}_{k_3}}.
\end{eqnarray}
For $\mathbf{e}\in \{\mathbf{e}_{1},\ldots, \mathbf{e}_{L}\}$, let
\begin{eqnarray*}
\eta_{k,\mathbf{e}}(\xi)=\left \{\begin{array}{ll} \displaystyle
\eta_{k}^{(2)}(\xi)\cdot \eta^{+}_{[k-5,k+5]}
(\xi\cdot \overline{\mathbf{e}})\text{ if } k\geq 100 ;\\
\eta_{k}^{(2)}(\xi)\hspace{2.9cm}\text{ if } k<100.
\end{array}
\right.
\end{eqnarray*}
For \eqref{n1e1} it suffices to prove that for any $\mathbf{e}\in
\{\mathbf{e}_{1},\ldots, \mathbf{e}_{L}\}$,
\begin{eqnarray}\label{n1e11}
&&2^{k_{1}+k_{2}+k_{3}}\|\eta_{k,\mathbf{e}}(\xi)(P(\tau,\xi)+i)^{-1}\mathcal{F}_{(2+1)}[P_{k}(\mathcal{R}(P_{k_1}u_1P_{k_2}
u_2)P_{k_3}u_3)]\|_{Z_{k}}\nonumber\\
&&\leq
C2^{\frac{k_{\max}-k}{2}}2^{\frac{3k_{\med}+3k_{\min}}{2}}(k_{\med}+1)^3(k_{\min}+1)^3\nonumber\\
&&\times
\|P_{k_1}u_1\|_{\widetilde{Z}_{k_1}}\|P_{k_2}u_2\|_{\widetilde{Z}_{k_2}}\|P_{k_3}u_3\|_{\widetilde{Z}_{k_3}}.
\end{eqnarray}
We first consider $k_1\geq k_{\max}-20$. So $k_1\geq k-25$. By an
angular partition of unity in frequency, we can assume
$\mathcal{F}(P_{k_1}u_1)$ is supported in the set
$$
\{(\xi,\tau): |\xi|\in[2^{k_1-1},2^{k_1+1}]\quad and \quad \xi\cdot
\overline{\mathbf{e}}_{0}\geq 2^{k_1-5}\}
$$
for some vector $\mathbf{e}_0\in \mathbb{S}^{1}$. Thus we have
\begin{eqnarray}\label{n1e22}
\|P_{k_1}u_1\|_{L^{\infty,2}_{\mathbf{e}_0}}\lesssim
2^{-k_1/2}\|P_{k_1}u_1\|_{\widetilde{Z}_{k_1}}.
\end{eqnarray}
By H\"older's inequality and \eqref{n1e22} we have
\begin{eqnarray}\label{n1e33}
&&2^{k_{1}+k_{2}+k_{3}}\|\eta_{k,\mathbf{e}}(\xi)(P(\tau,\xi)+i)^{-1}\mathcal{F}_{(2+1)}[P_{k}(\mathcal{R}(P_{k_1}u_1P_{k_2}
u_2)P_{k_3}u_3)]\|_{Z_{k}}\nonumber\\&\leq&
C2^{-\frac{k}{2}}2^{k_1+k_2+k_3}\|P_{k}(\mathcal{R}(P_{k_1}u_1P_{k_2}
u_2)P_{k_3}u_3)\|_{L^{1,2}_{\mathbf{e}}}\nonumber\\
&\leq&C2^{-\frac{k}{2}}2^{k_1+k_2+k_3}\|P_{k_3}u_3\|_{L^{2,\infty}_{\mathbf{e}}}\|P_{k_1}u_1P_{k_2}u_2\|_{L^{2}}\nonumber\\
&\leq&C2^{-\frac{k}{2}}2^{k_1+k_2+k_3}\|P_{k_3}u_3\|_{L^{2,\infty}_{\mathbf{e}}}\|P_{k_1}u_1\|
_{L^{\infty,2}_{\mathbf{e}_0}}\|P_{k_2}u_2\|_{L^{2,\infty}_{\mathbf{e}_0}}\nonumber\\
&\leq&C2^{-\frac{k}{2}+\frac{k_1}{2}+\frac{3k_2}{2}+\frac{3k_3}{2}}(k_2+1)^2(k_3+1)^2\|P_{k_1}u_1\|
_{\widetilde{Z}_{k_1}}\|P_{k_2}u_2\|_{\widetilde{Z}_{k_2}}\|P_{k_3}u_3\|_{\widetilde{Z}_{k_3}}
\end{eqnarray}
Which is enough for \eqref{n1e1}. The proof for $k_2\geq
k_{\max}-20$ is the same by symmetry.

Now, Let $k_{3}= k_{\max}$. In this case $k_3\geq k-3$. Furthermore,
in view of the above argument, we can assume that $k_3\geq
k_{\med}+20$, thus
\begin{eqnarray*}
&&\eta_{k,\mathbf{e}}(\xi)(P(\tau,\xi)+i)^{-1}\mathcal{F}_{(2+1)}[P_{k}(\mathcal{R}(P_{k_1}u_1P_{k_2}
u_2)P_{k_3}u_3)]\\
&=&\eta_{k,\mathbf{e}}(\xi)(P(\tau,\xi)+i)^{-1}\mathcal{F}_{(2+1)}[P_{k}(\mathcal{R}(P_{k_1}u_1P_{k_2}
u_2)\widetilde{P}_{k,\mathbf{e}}P_{k_3}u_3)]
\end{eqnarray*}
where
$\mathcal{F}(\widetilde{P}_{k,\mathbf{e}}f)(\xi,\tau)=\widetilde{\eta}_{k,\mathbf{e}}(\xi)\widehat{f}(\xi,\tau)$,
\begin{eqnarray*}
\widetilde{\eta}_{k,\mathbf{e}}(\xi)=\left \{\begin{array}{ll}
\displaystyle \eta_{[k-1,k+1]}^{(2)}(\xi)\cdot \eta_{[k-10,k+10]}
(\xi\cdot \overline{\mathbf{e}})\text{ if } k\geq 100. ;\\
\eta_{k}^{(2)}(\xi)\hspace{4.0cm}\text{ if } k<100.
\end{array}
\right.
\end{eqnarray*}
and
\begin{eqnarray}\label{n1e222}
\|\widetilde{P}_{k,\mathbf{e}}P_{k_3}u_3)\|_{L^{\infty,2}_{\mathbf{e}}}\lesssim
2^{-k_3/2}\|P_{k_3}u_3\|_{\widetilde{Z}_{k_3}}.
\end{eqnarray}
Thus
\begin{eqnarray*}
&&2^{k_{1}+k_{2}+k_{3}}\|\eta_{k,\mathbf{e}}(\xi)(P(\tau,\xi)+i)^{-1}\mathcal{F}_{(2+1)}
[P_{k}(\mathcal{R}(P_{k_1}u_1P_{k_2}
u_2)P_{k_3}u_3)]\|_{Z_{k}}\nonumber\\
&\leq&
C2^{-\frac{k}{2}}2^{k_{1}+k_{2}+k_{3}}\|P_{k}(\mathcal{R}(P_{k_1}u_1P_{k_2}
u_2)\widetilde{P}_{k,\mathbf{e}}P_{k_3}u_3)\|_{L^{1,2}_{\mathbf{e}}}\nonumber\\
&\leq&C2^{-\frac{k}{2}}2^{k_{1}+k_{2}+k_{3}}\|\widetilde{P}_{k,\mathbf{e}}P_{k_3}u_3\|_{L^{\infty,2}_{\mathbf{e}}}
\|\mathcal{R}(P_{k_1}u_1P_{k_2}
u_2)\|_{L^{1,\infty}_{\mathbf{e}}}\nonumber\\
&\leq&C2^{-\frac{k}{2}}2^{k_{1}+k_{2}+k_{3}}\|\widetilde{P}_{k,\mathbf{e}}P_{k_3}u_3\|_{L^{\infty,2}_{\mathbf{e}}}
\sum_{k\leq k_{\med}+1}\|P_k\mathcal{R}(P_{k_1}u_1P_{k_2}
u_2)\|_{L^{1,\infty}_{\mathbf{e}}}.
\end{eqnarray*}
By $\|P_{k}\mathcal{R}(f)\|_{L^{1,\infty}_{\mathbf{e}}}\leq
C\|f\|_{L^{1,\infty}_{\mathbf{e}}}$, we can continue with
\begin{eqnarray}\label{n1e333}
&&C2^{-\frac{k}{2}}2^{k_{1}+k_{2}+k_{3}}\|\widetilde{P}_{k,\mathbf{e}}P_{k_3}u_3\|_{L^{\infty,2}_{\mathbf{e}}}
(k_{\med}+1)\|P_{k_1}u_1P_{k_2}
u_2\|_{L^{1,\infty}_{\mathbf{e}}}\nonumber\\
&\leq&C2^{-\frac{k}{2}}2^{k_1+k_2+k_3}(k_{\med}+1)\|\widetilde{P}_{k,\mathbf{e}}P_{k_3}u_3\|_{L^{\infty,2}_{\mathbf{e}}}\|P_{k_1}u_1\|
_{L^{2,\infty}_{\mathbf{e}}}\|P_{k_2}u_2\|_{L^{2,\infty}_{\mathbf{e}}}\nonumber\\
&\leq&C2^{-\frac{k}{2}+\frac{k_3}{2}+\frac{3k_2}{2}+\frac{3k_1}{2}}(k_1+1)^3(k_2+1)^3\|P_{k_1}u_1\|
_{\widetilde{Z}_{k_1}}\|P_{k_2}u_2\|_{\widetilde{Z}_{k_2}}\|P_{k_3}u_3\|_{\widetilde{Z}_{k_3}}
\end{eqnarray}
We finish the proof of \eqref{n1e11}.
\end{proof}

\section{Multilinear Estimates}

The purpose of the this section is to estimate the nonlinear term
$$
\mathcal{N}_{0}(u)=\frac{2\overline{u}}{1+u\overline{u}}
$$
with $u\in C(\mathbb{R}: H^{\infty})$. The basic tool to analysis
the $\mathcal{N}_{0}(u)$ term is the algebra property of the
resolution spaces, say Lemma \ref{n2}.  In order to set up Lemma
\ref{n2}, we need the following two simple $L^{2}$ estimates.
\begin{lemma}\label{pl1}
If $k_{1},k_{2},k\in \mathbb{Z}_{+}$, $j_{1},j_{2},j\in
\mathbb{Z}_{+}$, and $g_{k_{1},j_{1}}, g_{k_{2},j_{2}}$ are $L^{2}$
functions supported in $D_{k_{1},j_{1}}$ and $D_{k_{2},j_{2}}$ then
\begin{eqnarray}\label{ple0}
\|1_{D_{k,j}}\cdot (g_{k_{1},j_{1}}*g_{k_{2},j_{2}})\|_{L^{2}}\leq
C2^{\min(k_1,k_2,k)}2^{\min(j_1,j_2,j)/2}\|g_{k_{1},j_{1}}\|_{L^{2}}\|g_{k_{2},j_{2}}\|_{L^{2}}.
\end{eqnarray}
\end{lemma}

For any $k,j\in \mathbb{Z}_{+}$, and $f_{k}\in Z_{k}$ we denote
$$
f_{k,\leq j}(\xi,\tau)=f_{k}(\xi,\tau)\cdot \eta_{\leq
j}(P(\tau,\xi))\text{ and }f_{k,\geq
j}(\xi,\tau)=f_{k}(\xi,\tau)\cdot \eta_{\geq j}(P(\tau,\xi)).
$$
We will use the following estimate in this section frequently.
\begin{lemma}\label{pl2}
If $k_{1},k_{2}\in \mathbb{Z}_{+}$, $k_{1}\leq
k_{2}+C$,$j_{1},j_{2}\in \mathbb{Z}_{+}$, and  $f_{k_{1}}\in
Z_{k_{1}}$,  $f_{k_{2}}\in Z_{k_{2}}$ and $\sigma' >1$ then
\begin{eqnarray}\label{ple1}
\|\widetilde{f}_{k_{1},\geq j_{1}}*\widetilde{f}_{k_{2},\geq
j_{2}}\|_{L^{2}}\leq
C(2^{j_{2}/2}2^{(k_{1}+k_{2})/2})^{-1}(2^{\sigma'
k_{1}}\|f_{k_{1}}\|_{Z_{k_{1}}})\|f_{k_{2}}\|_{Z_{k_{2}}}
\end{eqnarray}
where $\mathcal{F}^{-1}(\widetilde{f}_{k_{i},\geq j_{i}})\in
\{\mathcal{F}^{-1}(f_{k_{i},\geq
j_{i}}),\overline{\mathcal{F}^{-1}(f_{k_{i},\geq j_{i}})}\}$,
$i=1,2$.
\end{lemma}
\begin{proof}[\textbf{Proof of Lemma \ref{pl2}}]
If $k_{2}\leq 100$, By Lemma \ref{lp3} and Lemma \ref{lp6},
\begin{eqnarray*}
\|\widetilde{f}_{k_{1},\geq j_{1}}*\widetilde{f}_{k_{2},\geq
j_{2}}\|_{L^{2}}&\leq& C\|\mathcal{F}^{-1}(f_{k_{1},\geq
j_{1}})\|_{L^{\infty}}\|\mathcal{F}^{-1}(f_{k_{2},\geq
j_{2}})\|_{L^{2}}\\
&\leq&
C2^{k_{1}}\|f_{k_{1}}\|_{Z_{k_{1}}}2^{-j_{2}/2}\|f_{k_{2}}\|_{Z_{k_{2}}}.
\end{eqnarray*}
This is enough for \eqref{ple1}.

If $k_{2}\geq 100$, in view of Lemma \ref{lp3}, we can assume that:
$f_{k_{2}}$ is supported in $\{(\xi_{2},\tau_{2}): |\xi_{2}-v|\leq
2^{k_{2}-50}\}$ for some $v\in I_{k_{2}}^{(2)}$. Let
$\widehat{v}=\overline{v}/|v|$, then when $k_{1}+k_{2}\geq j_{2}$,
we use Lemma \ref{lp4}, Lemma \ref{lp5} and Lemma \ref{lp3} to get
\begin{eqnarray*}
\|\widetilde{f}_{k_{1},\geq j_{1}}*\widetilde{f}_{k_{2},\geq
j_{2}}\|_{L^{2}}&\leq& C\|\mathcal{F}^{-1}(f_{k_{1},\geq
j_{1}})\|_{L^{2,\infty}_{\widehat{v}}}\|\mathcal{F}^{-1}(f_{k_{2},\geq
j_{2}})\|_{L^{\infty,2}_{\widehat{v}}}\\
&\leq& C2^{k_{1}/2}(k_{1}+1)^2\|f_{k_{1},\geq
j_{1}}\|_{Z_{k_{1}}}2^{-k_{2}/2}\|f_{k_{2},\geq
j_{2}}\|_{Z_{k_{2}}}\\
&\leq& C2^{-(k_{1}+k_{2})/2}2^{\sigma'
k_{1}}\|f_{k_{1}}\|_{Z_{k_{1}}}\|f_{k_{2}}\|_{Z_{k_{2}}}.
\end{eqnarray*}
When $k_{1}+k_{2}\leq j_{2}$, we use the definition and Lemma
\ref{lp6} to get
\begin{eqnarray*}
\|\widetilde{f}_{k_{1},\geq j_{1}}*\widetilde{f}_{k_{2},\geq
j_{2}}\|_{L^{2}}&\leq& C\|\mathcal{F}^{-1}(f_{k_{1},\geq
j_{1}})\|_{L^{\infty}}\|\mathcal{F}^{-1}(f_{k_{2},\geq
j_{2}})\|_{L^{2}}\\
&\leq& C2^{k_{1}}\|f_{k_{1},\geq
j_{1}}\|_{Z_{k_{1}}}2^{-j_{2}/2}\|f_{k_{2},\geq
j_{2}}\|_{Z_{k_{2}}}\\
&\leq&
C2^{-j_{2}/2}2^{k_{1}}\|f_{k_{1}}\|_{Z_{k_{1}}}\|f_{k_{2}}\|_{Z_{k_{2}}}
\end{eqnarray*}
Thus we finish the proof.
\end{proof}

\begin{lemma}\label{n2}
Assume $u,v\in F^{\sigma}+\overline{F}^{\sigma}$, then for
$\sigma>1$ we have
\begin{eqnarray}\label{n2e0}
\|u\cdot v\|_{F^{\sigma}+\overline{F}^{\sigma}}\leq
C\|u\|_{F^{\sigma}+\overline{F}^{\sigma}}\|v\|_{F^{\sigma}+\overline{F}^{\sigma}}
\end{eqnarray}
\end{lemma}
\begin{remark}
If here we define $F^\sigma$ by $X_k$ instead of $Z_k$, then the
bilinear estimate \eqref{n2e0} was already proved in \cite{O}.
\end{remark}

\begin{proof}[\textbf{Proof of Lemma \ref{n2}}]
Let $f_{k}\in \{\eta_{k}^{(2)}(\xi)\cdot \mathcal{F}(u),
\eta_{k}^{(2)}(\xi)\cdot \mathcal{F}(\overline{u})\}$, and $g_{k}\in
\{\eta_{k}^{(2)}(\xi)\cdot \mathcal{F}(v), \eta_{k}^{(2)}(\xi)\cdot
\mathcal{F}(\overline{v})\}$. It suffices to show that for any
$k_{1},k_{2}\in \mathbb{Z}_{+}$, $\sigma >1$,
\begin{eqnarray}\label{n2e1}
&&\sum_{k\in\mathbb{Z}_{+}}2^{2\sigma
k}\Big(\sum_{k_1,k_2\in\mathbb{Z}_{+}}\|\eta_{k}^{(2)}(\xi)\cdot
(f_{k_{1}}* g_{k_{2}})\|_{Z_{k}}\Big)^{2}\nonumber\\
&\leq& C\Big(\sum_{k_1\in\mathbb{Z}_{+}}2^{2\sigma
k_{1}}\|f_{k_{1}}\|^{2}_{Z_{k_{1}}}\Big)\Big(\sum_{k_2\in\mathbb{Z}_{+}}2^{2\sigma
k_{2}}\|g_{k_{2}}\|^{2}_{Z_{k_{2}}}\Big).
\end{eqnarray}
Furthermore, we need to show that, if $k_{1},k_{2},k\in
\mathbb{Z}_{+}$, $k_{1}\leq k_{2}+10$, $f_{k_1}\in Z_{k_1}$ and
$f_{k_{2}}\in Z_{k_2}$, then
\begin{eqnarray}\label{n2e2}
2^{\sigma k}\|\eta_{k}^{(2)}(\xi)\cdot (\widetilde{f}_{k_{1}}*
f_{k_{2}})\|_{Z_{k}}\leq C2^{-|k_{2}-k|/4}(2^{\sigma'
k_{1}}\|f_{k_{1}}\|_{Z_{k_{1}}})(2^{\sigma
k_{2}}\|f_{k_{2}}\|_{Z_{k_{2}}}),
\end{eqnarray}
where $\mathcal{F}^{-1}(\widetilde{f}_{k_{1}})\in
\{\mathcal{F}^{-1}(f_{k_{1}}),\overline{\mathcal{F}^{-1}(f_{k_{1}})}\}$,
and $1<\sigma'<\sigma$.

We may assume $k\leq k_{2}+20$. If $k_{2}\leq 99$, the bound
\eqref{n2e2} follows easily from Lemma \ref{pl1} (also see the Case
1 below). We only consider the case $k_{2}\geq 100$. In view of
Lemma \ref{lp3}, we may assume that
$$
f_{k_{2}} \text{ is supported in }I_{k_{2}}^{(2)}\times
\mathbb{R}\cap \{(\xi_{2},\tau_{2}): |\xi_{2}-v|\leq 2^{k_{2}-50}\}
\text{for some } v\in I^{(2)}_{k_{2}}.
$$
With $v$ as above, let $\widehat{v}=\overline{v}/|v|\in
\mathbb{S}^{1}$ and
$$
\widetilde{K}=k_{1}+k_{2}+100
$$
By Lemma \ref{pl2} with $j_{1}=j_{2}=0$, we obtain
\begin{eqnarray*}
&&2^{\sigma k}\|\eta_{\leq
\widetilde{K}-1}(P(\tau,\xi))\eta_{k}^{(2)}(\xi)\cdot
(\widetilde{f}_{k_{1}}* f_{k_{2}})\|_{Z_{k}}\\
&\leq&2^{\sigma k}2^{\widetilde{K}/2}\|\widetilde{f}_{k_{1}}* f_{k_{2}}\|_{L^{2}}\\
&\leq&2^{\sigma k}(2^{\sigma'
k_{1}}\|f_{k_{1}}\|_{Z_{k_{1}}})\cdot\|f_{k_{2}}\|_{Z_{k_{2}}}
\end{eqnarray*}

So for \eqref{n2e2}, it remains to estimate
\begin{eqnarray}\label{n2e3}
&&2^{\sigma k}\|\eta_{\geq
\widetilde{K}}(P(\tau,\xi))\eta_{k}^{(2)}(\xi)\cdot
(\widetilde{f}_{k_{1}}* f_{k_{2}})\|_{Z_{k}}\nonumber\\
&\leq& C2^{-|k_{2}-k|/4}(2^{\sigma'
k_{1}}\|f_{k_{1}}\|_{Z_{k_{1}}})(2^{\sigma
k_{2}}\|f_{k_{2}}\|_{Z_{k_{2}}})
\end{eqnarray}
where  $1<\sigma'<\sigma$. By Lemma \ref{lp1}, we need to analyze
several cases.

\textbf{Case 1.} $f_{k_{2}}=g_{k_{2},j_{2}}\in X_{k_{2}}$,
$f_{k_{1}}\in Z_{k_{1}}$, let $g_{k_{1},j_{1}}=
\eta_{j_1}(P(\tau,\xi))f_{k_1}$. By the definition of $Z_k$ and
Lemma \ref{pl1}, we get
\begin{eqnarray}\label{n2e4}
&&2^{\sigma k}\|\eta_{\geq
\widetilde{K}}(P(\tau,\xi))\eta_{k}^{(2)}(\xi)\cdot
(\widetilde{g}_{k_{1},j_{1}}* f_{k_{2}})\|_{Z_{k}}\nonumber\\
&\leq&C2^{\sigma k}2^{\max(j_{1}+j_{2})/2}\sup_{j\leq
\max(j_{1},j_{2})+C}
\|1_{D_{k,j}}\cdot(\widetilde{g}_{k_{1},j_{1}}* g_{k_{2},j_{2}})\|_{L^{2}}\nonumber\\
&\leq&C2^{\sigma k}2^{\max(j_{1}+j_{2})/2}2^{
k_{1}}2^{\min(j_{1}+j_{2})/2} \|g_{k_{1},j_{1}}\|_{L^{2}}\|
g_{k_{2},j_{2}}\|_{L^{2}}.
\end{eqnarray}

\textbf{Case 2.} $f_{k_{2}}\in Y^{\mathbf{e}_{l}}_{k_{2}}$,
$f_{k_{1}}\in Z_{k_{1}}$, and $k_{2}\leq k_{1}+C$, so
$|k_{1}-k_{2}|\leq C$. By case 1, and Lemma \ref{lp3}, we can assume
that $f_{k_{2}}$ is supported in the set $\{(\xi_{2},\tau_{2}):
|P(\xi_{2},\tau_{2})|\leq 2^{\widetilde{K}-100}\}$. Thus $j_{1}\geq
\widetilde{K}-10$ (unless the left hand-side of \eqref{n2e3}
vanish). Let $g_{k_{1},j_{1}}= \eta_{j_1}(P(\tau,\xi))f_{k_1}$, we
have
\begin{eqnarray}\label{n2e5}
&&2^{\sigma k}\|\eta_{\geq
\widetilde{K}}(P(\tau,\xi))\eta_{k}^{(2)}(\xi)\cdot
(\widetilde{g}_{k_{1},j_{1}}* f_{k_{2}})\|_{Z_{k}}\nonumber\\
&\leq&C2^{\sigma k}2^{j_{1}/2}
\|\widetilde{g}_{k_{1},j_{1}}* f_{k_{2}}\|_{L^{2}}\nonumber\\
&\leq&C2^{\sigma k}2^{j_{1}/2}
\|g_{k_{1},j_{1}}\|_{L^{2}}\|\mathcal{F}^{-1}
(f_{k_{2}})\|_{L^{\infty}},
\end{eqnarray}
which is suffices for \eqref{n2e3} by Lemma \ref{lp6}.

\textbf{Case 3.} $f_{k_{2}}\in Y^{\mathbf{e}_{l}}_{k_{2}}$,
$f_{k_{1}}=f_{k_{1}}\cdot\eta_{\leq \widetilde{K}-1}(P(\tau,\xi))\in
Z_{k_{1}}$, $k_{1}\leq k_{2}-10$, so $|k-k_2|\leq 2$. It suffice to
prove
\begin{eqnarray}\label{n2c5} 2^{\sigma
k}\|\eta_{\geq \widetilde{K}}(P(\tau,\xi))\eta_{k}^{(2)}(\xi)\cdot
(\widetilde{f}_{k_{1}}* f_{k_{2}})\|_{Z_{k}}\nonumber\\
\leq C(2^{\sigma' k_{1}}\|f_{k_{1}}\|_{Z_{k_{1}}})(2^{\sigma
k_{2}}\|f_{k_{2}}\|_{Y^\mathbf{e}_{k_{2}}})
\end{eqnarray}
First notice that
$$
\widetilde{f}_{k_{1}}* f_{k_{2}} \text{ is supported in the set }
\{(\xi,\tau); \xi\cdot \overline{\mathbf{e}}_{l}\in
[2^{k-2},2^{k+2}] \}
$$
and the following identity
\begin{eqnarray}\label{null}
&&-\mathcal{F}^{-1}[(P(\tau,\xi)+i)\cdot(\widetilde{f}_{k_{1}}*
f_{k_{2}})]\nonumber\\
&=&(i\partial_{t}+\square-i)\mathcal{F}^{-1}(f_{k_{2}})\cdot
\mathcal{F}^{-1}(\widetilde{f}_{k_{1}})\nonumber\\
&&+2\nabla\mathcal{F}^{-1}(\widetilde{f}_{k_{1}})\cdot\widetilde{\nabla}
\mathcal{F}^{-1}(f_{k_{2}})\nonumber\\
&&+\mathcal{F}^{-1}(f_{k_{2}})\cdot(i\partial_{t}+\square)
\mathcal{F}^{-1}(\widetilde{f}_{k_{1}})
\end{eqnarray}
where $\widetilde{\nabla}=(\partial_{x_{1}},-\partial_{x_{2}})$, we
have
\begin{eqnarray}\label{n2e6}
&&2^{\sigma k}\|\eta_{\geq
\widetilde{K}}(P(\tau,\xi))\eta_{k}^{(2)}(\xi)\cdot
(\widetilde{f}_{k_{1}}* f_{k_{2}})\|_{Z_{k}}\nonumber\\
&\leq&C 2^{\sigma
k}2^{-k/2}\|(i\partial_{t}+\square-i)\mathcal{F}^{-1}(f_{k_{2}})\cdot
\mathcal{F}^{-1}(\widetilde{f}_{k_{1}})\|_{L^{1,2}_{\mathbf{e}_{l}}}\nonumber\\
&&+C 2^{\sigma
k}2^{-\widetilde{K}/2}\|\nabla\mathcal{F}^{-1}(\widetilde{f}_{k_{1}})\cdot\widetilde{\nabla}
\mathcal{F}^{-1}(f_{k_{2}})\|_{L^{2}}\nonumber\\
&&+C 2^{\sigma
k}2^{-\widetilde{K}/2}\|\mathcal{F}^{-1}(f_{k_{2}})\cdot(i\partial_{t}+\square)
\mathcal{F}^{-1}(\widetilde{f}_{k_{1}})\|_{L^{2}}.
\end{eqnarray}
The first term in the right-hand side of \eqref{n2e6} can be
controlled by
$$
C 2^{\sigma
k}2^{-k/2}\|(i\partial_{t}+\square-i)\mathcal{F}^{-1}(f_{k_{2}})\|_{L^{1,2}_{\mathbf{e}_{l}}}\cdot\|
\mathcal{F}^{-1}(\widetilde{f}_{k_{1}})\|_{L^{\infty}},
$$
which is enough for \eqref{n2c5} in view of Lemma \ref{lp6}. The
second and the third terms are bounded by
$$
C 2^{\sigma
k}2^{-\widetilde{K}/2}2^{\widetilde{K}}\|\mathcal{F}^{-1}(\widetilde{f}_{k_{1}})\cdot
\mathcal{F}^{-1}(f_{k_{2}})\|_{L^{2}},
$$
Which is enough for \eqref{n2c5} by Lemma \ref{pl2}.

\textbf{Case 4.} $f_{k_{2}}\in Y^{\mathbf{e}_{l}}_{k_{2}}$,
$f_{k_{1}}\in Z_{k_{1}}$, $k_{1}\leq k_{2}-10$ and $j_{1}\geq
\widetilde{K}$. we denote $g_{k_{1},j_{1}}=
\eta_{j_1}(P(\tau,\xi))f_{k_1}$, for \eqref{n2e3}, it suffices to
prove
\begin{eqnarray}\label{n2c6} 2^{\sigma
k}\|\eta_{\geq \widetilde{K}}(P(\tau,\xi))\eta_{k}^{(2)}(\xi)\cdot
(\widetilde{g}_{k_{1},j_{1}}* f_{k_{2}})\|_{Z_{k}} \leq C2^{
k_{1}}2^{j_1/2}\|g_{k_{1},j_{1}}\|_{L^{2}}(2^{\sigma
k_{2}}\|f_{k_{2}}\|_{Y^\mathbf{e}_{k_{2}}})
\end{eqnarray}
By Lemma \ref{lp1}, we decompose
\begin{equation*}
f_{k_2}=f_{k_2,\leq j_1-10}+f_{k_2,\geq j_1+10}+X_{k_2}.
\end{equation*}
In view of Case 1, for \eqref{n2c6}, it suffices to prove that
\begin{equation}\label{bt95}
\begin{split}
&2^{\sigma
k}\|\eta_{k,\mathbf{e}}^{(2)}(\xi)\cdot (\widetilde{g}_{k_1,j_1}\ast f_{k_2,\leq j_1-10})\|_{Z_k}\\
&+2^{\sigma k}\|\eta_{\geq j_1}(P(\tau,\xi))\cdot
\eta_{k,\mathbf{e}}^{(2)}(\xi)\cdot
(\widetilde{g}_{k_1,j_1}\ast f_{k_2,\geq  j_1+10})\|_{Z_k}\\
&\leq C(2^{ k_1}2^{j_1/2}\|g_{k_1,j_1}\|_{L^2})\cdot (2^{\sigma
k_2}\|f_{k_2}\|_{Y^\mathbf{e}_{k_2}}).
\end{split}
\end{equation}
For the first term in \eqref{bt95}, it suffices to prove
\begin{eqnarray}\label{n2e8}
2^{\sigma k}2^{j_{1}/2}\|f_{k_{2},\leq
j_{1}-10}*\widetilde{g}_{k_{1},j_{1}}\|_{L^{2}}\leq
C(2^{k_{1}}2^{j_{1}/2}\|g_{k_{1},j_{1}}\|_{L^{2}})\cdot(2^{\sigma
k_{2}}\|f_{k_{2}}\|_{Y_{k_{2}}^{\e^{l}}}).
\end{eqnarray}
By Lemma \ref{lp4*}, we can assume that
\begin{eqnarray*}
f_{k_{2},\leq j_{1}-10}(\xi,\tau)=2^{-k_2/2}\frac{\eta_{\leq
k_2-90}(\xi_{1}^\mathbf{e}-t^*_\mathbf{e})\chi_{k_2,5}(M^*_\e)}{\xi_{1}^\mathbf{e}-t^*_\mathbf{e}+i/2^{k_2}}
h(\xi_{2}^\mathbf{e},\tau),
\end{eqnarray*}
where $\|h\|_{L^{2}}\leq C\|f_{k_{2},\leq
j_{1}-10}\|_{Y_{k_2}^\mathbf{e}}$. For \eqref{n2e8}, it suffices to
prove
\begin{eqnarray}\label{n2e9}
\|f_{k_{2},\leq j_{1}-10}*\widetilde{g}_{k_{1},j_{1}}\|_{L^{2}}\leq
C2^{k_{1}}\|g_{k_{1},j_{1}}\|_{L^{2}}\cdot\|h\|_{L^{2}}.
\end{eqnarray}
We estimate the $L^2$ norm in the left-hand side of \eqref{n2e9} by
duality. The left-hand side of \eqref{n2e9} is bounded by
\begin{eqnarray}\label{DU}
I&=&2^{-k_2/2}\sup_{ \|a\|_{L^2}=1}\Big|\int_{\mathbb{R}^6}
\widetilde{g}_{k_1,j_1}(\eta_1^\e \e+\eta_2^\e \mathbf{e}^{\perp},\beta)\cdot h(\xi_2^\e,\tau)\nonumber\\
&&\times\frac{\eta_{\leq
k_2-90}(\xi_{1}^\mathbf{e}-t^*_\mathbf{e})\chi_{k_2,5}(M^*_\e(\tau,\xi_{2}^\mathbf{e}))}{\xi_{1}^\mathbf{e}-t^*_\mathbf{e}+i/2^{k_2}}\nonumber\\
&&\times
a(\xi_1^\e+\eta_1^\e,\xi_2^\e+\eta_2^\e,\tau+\beta)\,d\xi_1^\e
d\eta_1^\e d\xi_2^\e d\eta_2^\e d\tau
d\beta\Big|\nonumber\\
&=&2^{-k_2/2}\sup_{ \|a\|_{L^2}=1}\Big|\int_{\mathbb{R}^5}
\widetilde{g}_{k_1,j_1}(\eta_1^\e \e+\eta_2^\e \mathbf{e}^{\perp},\beta)\cdot h(\xi_2^\e,\tau)\\
&&\widetilde{a}(\eta_1^\e+t^*_\mathbf{e},\xi_2^\e+\eta_2^\e,\tau+\beta)d\eta_1^\e
d\xi_2^\e d\eta_2^\e d\tau d\beta\Big|.\nonumber
\end{eqnarray}
Here
$$\widetilde{a}(\eta_1^\e,\eta_2^\e,\beta)=\int_{\R}\frac{\eta_{\leq
k_2-90}(\xi_{1}^\mathbf{e})\chi_{k_2,5}(M^*_\e(\tau,\xi_{2}^\mathbf{e}))}
{\xi_{1}^\mathbf{e}+i/2^{k_2}}\cdot
a(\xi_1^\e+\eta_1^\e,\eta_2^\e,\beta)\,d\xi_1^\e.
$$
The boundedness of Hilbert transform gives
$\|\widetilde{a}(\eta_1^\e,\eta_2^\e,\beta)\|_{L^{2}_{\eta_1^\e}}\leq
C\|a(\eta_1^\e,\eta_2^\e,\beta)\|_{L^{2}_{\eta_1^\e}}$. By using
H\"{o}lder's ine\-qua\-li\-ty in the variables
$(\xi_2^\e,\tau,\beta)$, we get
\begin{eqnarray}\label{la1}
I&\leq&C2^{-k_2/2}\sup_{ \|a\|_{L^2}=1}\int_{\mathbb{R}^2}
\|\widetilde{g}_{k_1,j_1}(\eta_1^\e \e+\eta_2^\e
\mathbf{e}^{\perp},\beta)\|_{L^2_\beta}\cdot
\|h(\xi_2^\e,\tau)\|_{L^{2}_{\xi_2^\e,\tau}}\nonumber\\
&&\|\widetilde{a}(\eta_1^\e+t^*_\mathbf{e},\xi_2^\e,\beta)\|
_{L^2_{\xi_2^\e,\tau,\beta}}d\eta_1^\e d\eta_2^\e.
\end{eqnarray}
Here $t^*_\mathbf{e}$ is the same as Lemma \ref{lp4*}, so we have $
|\partial_{\tau}t^*_\mathbf{e}|\geq c2^{-k_2} $. Then by change of
variables, we have
\begin{eqnarray*}
I&\leq&C2^{-k_2/2}2^{k_2/2}\int_{\mathbb{R}^2}
\|\widetilde{g}_{k_1,j_1}(\eta_1^\e \e+\eta_2^\e
\mathbf{e}^{\perp},\beta)\|_{L^2_\beta} d\eta_1^\e d\eta_2^\e\cdot
\|h(\xi_2^\e,\tau)\|_{L^{2}_{\xi_2^\e,\tau}}
\end{eqnarray*}
which is sufficient for \eqref{n2e9}.

From \eqref{null}, we can control the second term in the right-hand
side of \eqref{bt95} by
\begin{equation}\label{pp8}
\begin{split}
&C2^{\sigma k}2^{-k/2}||(i\partial_t+\square_x-i)
\mathcal{F}^{-1}(f_{k_2,\geq j_1+10})\cdot \mathcal{F}^{-1}(\widetilde{g}_{k_1,j_1})||_{L^{1,2}_{\mathbf{e}}}\\
&+C 2^{\sigma k}2^{-j_1/2}||\mathcal{F}^{-1}(f_{k_2,\geq
j_1+10})\cdot
(i\partial_t+\square_x)\mathcal{F}^{-1}(\widetilde{g}_{k_1,j_1})||_{L^2}\\
&+C 2^{\sigma k}2^{-j_1/2}||\nabla_x\mathcal{F}^{-1}(f_{k_2,\geq
j_1+10})\cdot
\widetilde{\nabla}_x\mathcal{F}^{-1}(\widetilde{g}_{k_1,j_1})||_{L^2}.
\end{split}
\end{equation}
We estimate the first term in the right-hand side of \eqref{pp8} by
\begin{equation*}
C2^{\sigma
k}2^{-k/2}||(i\partial_t+\square_x-i)\mathcal{F}^{-1}(f_{k_2,\geq
j_1+10})\|_{L^{1,2}_{\mathbf{e}}}\cdot
\|\mathcal{F}^{-1}(\widetilde{g}_{k_1,j_1})||_{L^\infty},
\end{equation*}
which is bounded by the right-hand side of \eqref{bt95} in view of
Lemma \ref{lp6}. We estimate the last two terms in the right-hand
side of \eqref{pp8} by
\begin{eqnarray*}
&&C 2^{\sigma k}2^{-j_1/2}\cdot 2^{j_1}\|f_{k_2,\geq j_1+10}\ast
\widetilde{g}_{k_1,j_1}||_{L^2}\\
&\leq&C 2^{\sigma k}2^{j_1/2}\|f_{k_2,\geq j_1+10}\|_{L^2}
\|\mathcal{F}_{(2+1)}^{-1}(\widetilde{g}_{k_1,j_1})||_{L^\infty},
\end{eqnarray*}
which is bounded by the right-hand side of \eqref{bt95} in view of
Lemma \ref{lp6}.
\end{proof}

From Lemma \ref{n2} and \eqref{em}, we have
\begin{corollary}\label{c1}
If $\sigma >1$ and $u_{1},\cdots, u_{n}\in F^{\sigma}$, then the
product $\widetilde{u}_{1}\cdot \ldots\cdot \widetilde{u}_{n}\in
\widetilde{F}^{\sigma}$ and
$$
\|\widetilde{u}_{1}\cdot \ldots\cdot
\widetilde{u}_{n}\|_{\widetilde{F}^{\sigma}}\leq
(C_{\sigma})^{n}\cdot \|u_{1}\|_{F^{\sigma}}\cdot \ldots \cdot
\|u_{n}\|_{F^{\sigma}},
$$
where $\widetilde{u}_{m}\in \{u_{m},\overline{u}_{m}\}$ for
$m=1,\ldots,n$.
\end{corollary}

Now we analyze the term
$$
\mathcal{N}_{0}(u)=\frac{2\bar{u}}{1+u\bar{u}} \in  C(\mathbb{R};
H^{\infty})
$$
with $u\in C(\mathbb{R}; H^{\infty})$.
\begin{lemma}\label{n4}
For $\sigma>1$ then there is $c(\sigma)>0$ with the property that
\begin{eqnarray}\label{ne5}
\|J^{\sigma'}(\mathcal{N}_{0}(u)-\mathcal{N}_{0}(v))\|_{\widetilde{F}^{\sigma}}\leq
C(\sigma, \sigma',
\|u\|_{F^{\sigma+\sigma'}}+\|v\|_{F^{\sigma+\sigma'}})\|J^{\sigma'}(u-v)\|_{F^{\sigma}}
\end{eqnarray}
for any $\sigma'\in\mathbb{Z}_{+}$, and $u,v \in
B_{F^{\sigma_{0}}}(0,c(\sigma))\cap F^{\sigma}$.
\end{lemma}
\begin{proof}[\textbf{Proof of Lemma \ref{n4}}]
We write first
$$
\mathcal{N}_{0}(u)-\mathcal{N}_{0}(v)=\frac{\overline{u-v}}{(1+u\bar{u})(1+v\bar{v})}-
\frac{(u-v)\overline{uv}}{(1+u\bar{u})(1+v\bar{v})}
$$
First we expand the above to power series, then by Corollary
\ref{c1} we can get \eqref{ne5} when $c(\sigma)$ sufficently small.
\end{proof}

\section{Proof of Theorem \ref{t2}}

In this section we prove Theorem \ref{t2}. Our main ingredients are
Lemma \ref{le1}, Lemma \ref{le2}, Lemma \ref{n1}, Lemma \ref{n4},
and the bound
\begin{equation}\label{fi1}
\sup_{t\in\mathbb{R}}\|u\|_{H^\sigma}\leq
C_\sigma\|u\|_{F^\sigma}\text{ for any }\sigma\geq 0\text{ and }u\in
F^\sigma,
\end{equation}
which follows from Lemma \ref{lp6}. Assume that ${\sigma_0}>3/2$ and
$\phi\in H^\infty\cap B_{H^{\sigma_0}}(0,\epsilon({\sigma_0}))$,
where $\epsilon({\sigma_0})\ll 1$ is to be fixed. We define
recursively
\begin{equation}\label{fi2*}
\begin{cases}
&u_0=\psi(t)\cdot W(t)\phi;\\
&u_{n+1}=\psi(t)\cdot W(t)\phi+\psi(t)\cdot
\int_0^tW(t-s)(\mathcal{N}(u_n(s)))\,ds\text{ for }n\in\mathbb{Z}_+.
\end{cases}
\end{equation}
Where $\mathcal{N}$ defined in \eqref{N}, that is
\begin{eqnarray*}\label{F}
\mathcal{N}(u_n)
&=&\psi(t)\cdot\big[\mathcal{N}_{0}(u_n)[(\partial_{x_{1}}u_n)^{2}-
(\partial_{x_{2}}u_n)^{2}]-4b\mathcal{R}_1(\mathcal{N}_{0}(u_n)\partial_{x_1}
u_n)\partial_{x_2} u_n\nonumber\\
&&+4b\mathcal{R}_2(\mathcal{N}_{0}(u_n)\partial_{x_2}
u_n)\partial_{x_2}
u_n-4b\mathcal{R}_3(\mathcal{N}_{0}(u_n)\partial_{x_1}
u_n)\partial_{x_1}
u_n\\
&&+4b\mathcal{R}_1(\mathcal{N}_{0}(u_n)\partial_{x_2}
u_n)\partial_{x_1} u_n\big]\nonumber.
\end{eqnarray*}

The rest of the proof is organized as follows. We first analyze
\eqref{fi2} with $\mathcal{N}(u_n)$ replaced by
$\psi(t)\mathcal{N}_0(u_n)(\partial_{x_1}u_n)^2$, then notice that
all the results hold for the $\mathcal{N}(u_n)$ case. Finally, we
use these results to conclude the proof of Theorem \ref{t2}.

Now we define recursively
\begin{equation}\label{fi2}
\begin{cases}
&u_0=\psi(t)\cdot W(t)\phi;\\
&u_{n+1}=\psi(t)\cdot W(t)\phi+\psi(t)\cdot
\int_0^tW(t-s)(\widetilde{\mathcal{N}}(u_n(s)))\,ds\text{ for
}n\in\mathbb{Z}_+,
\end{cases}
\end{equation} where $\widetilde{\mathcal{N}}(u_n(s))=\psi(s)\mathcal{N}_0(u_n(s))(\partial_{x_1}u_n(s))^2$, clearly, $u_n\in
C(\mathbb{R}:H^\infty)$.

We show first that
\begin{equation}\label{fi3}
\|u_n\|_{F^{\sigma_0}}\leq C_{\sigma_0}\|\phi\|_{H^{\sigma_0}}\text{
for any }n=0,1,\ldots,\text{ if }\epsilon({\sigma_0}) \text{ is
sufficiently small}.
\end{equation}
The bound \eqref{fi3} holds for $n=0$, due to Lemma \ref{le1}. Then,
using Lemma \ref{n4} with $\sigma'=0$, $v\equiv 0$, Lemma \ref{n1},
and the inequality \eqref{em}, we have
\begin{equation*}
\|\widetilde{\mathcal{N}}(u_n)\|_{N^{\sigma_0}}\leq
C_{\sigma_0}\|u_n\|_{F^{\sigma_0}}^3.
\end{equation*}
Using Lemma \ref{le2}, the definition \eqref{fi2}, and Lemma
\ref{le1}, it follows that
\begin{equation*}
\|u_{n+1}\|_{F^{\sigma_0}}\leq
C_{\sigma_0}\|\phi\|_{H^{\sigma_0}}+C_{\sigma_0}\|u_n\|_{F^{\sigma_0}}^3,
\end{equation*}
which leads to \eqref{fi3} by induction over $n$.

We show now that
\begin{equation}\label{fi4}
\|u_{n}-u_{n-1}\|_{F^{\sigma_0}}\leq 2^{-n}\cdot
C_{\sigma_0}\|\phi\|_{H^{\sigma_0}}\text{ for any
}n\in\mathbb{Z}_+\text{ if }\epsilon({\sigma_0}) \text{ is
sufficiently small}.
\end{equation}
This is clear for $n=0$ (with $u_{-1}\equiv 0$), by Lemma \ref{le1}.
Then, using Lemma \ref{n4} with $\sigma'=0$, Lemma \ref{n1}, and the
estimates \eqref{em} and \eqref{fi3}, we have
\begin{equation*}
\|\widetilde{\mathcal{N}}(u_{n-1})-\widetilde{\mathcal{N}}(u_{n-2})\|_{N^{\sigma_0}}\leq
C_{\sigma_0}\cdot \epsilon({\sigma_0})^2\cdot \|u_{n-1}-
u_{n-2}\|_{F^{\sigma_0}}.
\end{equation*}
Using Lemma \ref{le2} and the definition \eqref{fi2} it follows that
\begin{equation*}
\|u_{n}-u_{n-1}\|_{F^{\sigma_0}}\leq C_{\sigma_0}\cdot
\epsilon({\sigma_0})^2\cdot \|u_{n-1}-u_{n-2}\|_{F^{\sigma_0}},
\end{equation*}
which leads to \eqref{fi4} by induction over $n$.

We show now that
\begin{equation}\label{fi5}
\|J^{\sigma'}(u_n)\|_{F^{\sigma_0}}\leq
C({\sigma_0},\sigma',\|J^{\sigma'}\phi\|_{H^{\sigma_0}})\text{ for
any }n,\sigma'\in\mathbb{Z}_+.
\end{equation}
We argue by induction over $\sigma'$ (the case $\sigma'=0$ follows
from \eqref{fi3}). So we may assume that
\begin{equation}\label{fi6}
\|J^{\sigma'-1}(u_n)\|_{F^{\sigma_0}}\leq
C({\sigma_0},\sigma',\|J^{\sigma'-1}\phi\|_{H^{\sigma_0}})\text{ for
any }n\in\mathbb{Z}_+,
\end{equation}
and it suffices to prove that
\begin{equation}\label{fi7}
\|\partial_{x_i}^{\sigma'}(u_n)\|_{F^{\sigma_0}}\leq
C({\sigma_0},\sigma',\|J^{\sigma'}\phi\|_{H^{\sigma_0}})\text{ for
any }n\in\mathbb{Z}_+\text{ and }i=1,2.
\end{equation}
The bound \eqref{fi7} for $n=0$ follows from Lemma \ref{le1}. We use
the decomposition
\begin{eqnarray*}\label{F}
\widetilde{\mathcal{N}}(u_n)
&=&\psi(t)\cdot\mathcal{N}_{0}(u_n)(\partial_{x_{1}}u_n)^{2},
\end{eqnarray*}
thus
\begin{equation}\label{fi8}
\partial_{x_i}^{\sigma'}(\mathcal{N}(u_n))=2\psi(t)\cdot\mathcal{N}_0(u_n)\cdot\partial_{x_1}u_n\cdot\partial_{x_i}^{\sigma'}\partial_{x_1}u_n+E_n,
\end{equation}
where
\begin{equation*}
E_n=\psi(t)\cdot \sum_{\sigma'_1+\sigma'_2+\sigma'_3=\sigma'\text{
and
}\sigma'_3,\sigma'_2<\sigma'}\partial_{x_i}^{\sigma'_1}\mathcal{N}_0(u_n)\cdot\partial_{x_i}^{\sigma'_2}\partial_{x_1}
u_n\cdot\partial_{x_i}^{\sigma'_3}\partial_{x_1}u_n.
\end{equation*} Using  Lemma \ref{n1},
\begin{equation*}
\begin{split}
||E_n||_{N^{\sigma_0}}\leq C_{\sigma_0}&\sum_{\sigma'_1+\sigma'_2+\sigma'_3=\sigma'\text{ and }\sigma'_3,\sigma'_2<\sigma'}\\
&||J^{-1}\partial_{x_i}^{\sigma'_1}\mathcal{N}_0(u_n)||_{\widetilde{F}^{\sigma_0}}\cdot
||J^{-1}\partial_{x_i}^{\sigma'_2}\partial_{x_1}u_n||_{\widetilde{F}^{\sigma_0}}\cdot
||J^{-1}\partial_{x_i}^{\sigma'_3}\partial_{x_1}u_n||_{\widetilde{F}^{\sigma_0}}.
\end{split}
\end{equation*}
Using now Lemma \ref{n4} with $v=0$, the bound \eqref{em}, and the
induction hypothesis \eqref{fi6}, we have
\begin{equation}\label{fi9}
||E_n||_{N^{\sigma_0}}\leq
C({\sigma_0},\sigma',\|J^{\sigma'-1}\phi\|_{H^{\sigma_0}}).
\end{equation}
In addition, using again Lemma \ref{n1}, Lemma \ref{n4} with $v=0$,
\eqref{em} and \eqref{fi3},
\begin{equation}\label{fi10}
||\psi(t)\cdot\mathcal{N}_0(u_n)\cdot\partial_{x_1}u_n\cdot\partial_{x_i}^{\sigma'}\partial_{x_1}u_n||_{N^{\sigma_0}}\leq
C_{{\sigma_0}}\cdot \epsilon({\sigma_0})^2\cdot
||\partial_{x_i}^{\sigma'}u_n||_{F^{\sigma_0}}.
\end{equation}
We use now the definition \eqref{fi2}, together with Lemma
\ref{le1}, Lemma \ref{le2}, and the bounds \eqref{fi9} and
\eqref{fi10} to conclude that
\begin{equation*}
||\partial_{x_i}^{\sigma'}u_{n+1}||_{F^{\sigma_0}}\leq
C({\sigma_0},\sigma',\|J^{\sigma'}\phi\|_{H^{\sigma_0}})+C_{{\sigma_0}}\cdot
\epsilon({\sigma_0})^2\cdot
||\partial_{x_i}^{\sigma'}u_n||_{F^{\sigma_0}}.
\end{equation*}
The bound  \eqref{fi7} follows by induction over $n$ provided that
$\epsilon({\sigma_0})$ is sufficiently small.

Finally, we show that
\begin{equation}\label{fi20}
\|J^{\sigma'}(u_{n}-u_{n-1}))\|_{F^{\sigma_0}}\leq 2^{-n}\cdot
C({\sigma_0},\sigma',\|J^{\sigma'}\phi\|_{H^{\sigma_0}})\text{ for
any }n,\sigma'\in\mathbb{Z}_+.
\end{equation}
As before, we argue by induction over $\sigma'$ (the case
$\sigma'=0$ follows from \eqref{fi4}). So we may assume that
\begin{equation}\label{fi61}
\|J^{\sigma'-1}(u_n-u_{n-1})\|_{F^{\sigma_0}}\leq 2^{-n}\cdot
C({\sigma_0},\sigma',\|J^{\sigma'-1}\phi\|_{H^{\sigma_0}})\text{ for
any }n\in\mathbb{Z}_+,
\end{equation}
and it suffices to prove that
\begin{equation}\label{fi71}
\|\partial_{x_i}^{\sigma'}(u_n-u_{n-1})\|_{F^{\sigma_0}}\leq
2^{-n}\cdot
C({\sigma_0},\sigma',\|J^{\sigma'}\phi\|_{H^{\sigma_0}})\text{ for
any }n\in\mathbb{Z}_+\text{ and }i=1, 2.
\end{equation}
The bound \eqref{fi71} for $n=0$ follows from Lemma \ref{le1}. For
$n\geq 1$ we use the decomposition
\begin{equation}\label{fi73}
\begin{split}
\mathcal{N}(u_{n-1})-&\mathcal{N}(u_{n-2})=\psi(t)\cdot(\mathcal{N}_0(u_{n-1})-\mathcal{N}_0(u_{n-2}))\cdot(\partial_{x_1}u_{n-1})^2\\
&+\psi(t)\cdot\mathcal{N}_0(u_{n-2})\cdot\partial_{x_1}(u_{n-1}-u_{n-2})\cdot
\partial_{x_1}(u_{n-1}+u_{n-2}).
\end{split}
\end{equation}
The same argument as before, which consists of expanding the
$\sigma'$ derivative, and combining Lemma \ref{n1}, Lemma \ref{n4},
\eqref{fi5}, and \eqref{fi61}, shows that
\begin{equation}\label{fi72}
\begin{split}
\big|\big|\partial_{x_i}^{\sigma'}\big[\psi(t)\cdot(\mathcal{N}_0(u_{n-1})-\mathcal{N}_0(u_{n-2}))\cdot(\partial_{x_1}u_{n-1})^2\big]\big|\big|_{N^{\sigma_0}}\\
\leq 2^{-n}\cdot
C({\sigma_0},\sigma',||J^{\sigma'}\phi||_{H^{\sigma_0}}),
\end{split}
\end{equation}
To estimate the $\sigma'$ derivative of the term in the second line
of \eqref{fi73}, we expand again the $\sigma'$ derivatives. Using
again the combination of Lemma \ref{n1}, Lemma \ref{n4},
\eqref{fi5}, and \eqref{fi61}, the $N^{\sigma_0}$ norm of most of
the terms that appear is again dominated by $2^{-n}\cdot
C({\sigma_0},\sigma',||J^{\sigma'}\phi||_{H^{\sigma_0}})$. The only
remaining terms are
\begin{equation*}
\psi(t)\cdot\mathcal{N}_0(u_{n-2})\cdot\partial_{x_i}^{\sigma'}\partial_{x_1}(u_{n-1}-u_{n-2})\cdot
\partial_{x_1}(u_{n-1}+u_{n-2}),
\end{equation*}
and we can estimate
\begin{equation*}
\begin{split}
||\psi(t)\cdot\mathcal{N}_0(u_{n-2})\cdot\partial_{x_i}^{\sigma'}\partial_{x_1}(u_{n-1}-u_{n-2})\cdot \partial_{x_1}(u_{n-1}+u_{n-2})||_{N^{\sigma_0}}\\
\leq C_{\sigma_0}\cdot\epsilon({\sigma_0})^2\cdot
||\partial_{x_1}^{\sigma'}(u_{n-1}-u_{n-2})||_{F^{\sigma_0}}.
\end{split}
\end{equation*}
As before, it follows that
\begin{equation*}
\begin{split}
||\partial_{x_i}^{\sigma'}(u_{n}-u_{n-1})||_{F^{\sigma_0}}&\leq 2^{-n}\cdot C({\sigma_0},\sigma',\|J^{\sigma'}\phi\|_{H^{\sigma_0}})\\
&+C_{{\sigma_0}}\cdot \epsilon({\sigma_0})^2\cdot
||\partial_{x_i}^{\sigma'}(u_{n-1}-u_{n-2})||_{F^{\sigma_0}}.
\end{split}
\end{equation*}
The bound  \eqref{fi71} follows by induction provided that
$\epsilon({\sigma_0})$ is sufficiently small.

In view of Lemma \ref{n1}, we notice that $\mathcal{N}(u)$ and
$\widetilde{\mathcal{N}}(u)$ share the same nonlinear estimate, so
the argument above for system \eqref{fi2} can be used to system
\eqref{fi2*}. Thus, \eqref{fi3}, \eqref{fi4}, \eqref{fi5},and
\eqref{fi20} also hold for $u_n$ defined by system \eqref{fi2*}.

We can now use \eqref{fi20} and \eqref{fi1} to construct
\begin{equation*}
u=\lim_{n\to\infty}u_n\in C(\mathbb{R}:H^\infty).
\end{equation*}
In view of \eqref{fi2},
\begin{equation*}
u=\psi(t)\cdot W(t)\phi+\psi(t)\cdot
\int_0^tW(t-s)(\mathcal{N}(u(s)))\,ds\text{ on
}\mathbb{R}^d\times\mathbb{R},
\end{equation*}
so $\widetilde{S}^\infty(\phi)$, the restriction of $u$ to
$\mathbb{R}^d\times[-1,1]$, is a solution of the initial-value
problem \eqref{t21}.

For Theorem \ref{t2} (b) and (c), it suffices to show that if
$\sigma'\in\mathbb{Z}_+$ and $\phi,\phi'\in
B_{H^{\sigma_0}}(0,\epsilon({\sigma_0}))\cap H^\infty$ then
\begin{equation}\label{fi80}
\sup_{t\in[-1,1]}||\widetilde{S}^\infty(\phi)-\widetilde{S}^\infty(\phi')||_{H^{\sigma_0+\sigma'}}\leq
C(\sigma_0,\sigma',||\phi||_{H^{\sigma_0+\sigma'}})\cdot
||\phi-\phi'||_{H^{\sigma_0+\sigma'}}.
\end{equation}
Part  (b) corresponds to the case $\sigma'=0$. To prove
\eqref{fi80}, we define the sequences $u_n$ and $u'_n$,
$n\in\mathbb{Z}_+$, as in \eqref{fi2}. Using  Lemma \ref{le1},
\begin{equation*}
||u_0-u'_0||_{F^{\sigma_0}}\leq
C_{\sigma_0}||\phi-\phi'||_{H^{\sigma_0}}.
\end{equation*}
Then we decompose $\mathcal{N}(u_n)-\mathcal{N}(u'_n)$ in the same
way as in \eqref{fi73}. As before, we combine Lemma \ref{le1}, Lemma
\ref{le2}, Lemma \ref{n1}, Lemma \ref{n4}, and the uniform bound
\eqref{fi3} to conclude that
\begin{equation*}
||u_{n+1}-u'_{n+1}||_{F^{\sigma_0}}\leq
C_{\sigma_0}||\phi-\phi'||_{H^{\sigma_0}}+C_{\sigma_0}\cdot\epsilon({\sigma_0})^2\cdot
||u_{n}-u'_{n}||_{F^{\sigma_0}}.
\end{equation*}
By induction over $n$ it follows that
\begin{equation*}
||u_n-u'_n||_{F^{\sigma_0}}\leq
C_{\sigma_0}||\phi-\phi'||_{H^{\sigma_0}}\text{ for any
}n\in\mathbb{Z}_+.
\end{equation*}
In view of \eqref{fi1} this proves \eqref{fi80} for  $\sigma'=0$.

Assume now that $\sigma'\geq  1$. In view of \eqref{fi1}, for
\eqref{fi80} it suffices to prove that
\begin{equation}\label{fi81}
||J^{\sigma'}(u_n-u'_n)||_{F^{\sigma_0}}\leq
C(\sigma_0,\sigma',||J^{\sigma'}(\phi)||_{H^{\sigma_0}})\cdot
||J^{\sigma'}(\phi-\phi')||_{H^{\sigma_0}},
\end{equation}
for any $n\in\mathbb{Z}_+$. We argue, as before, by induction over
$\sigma'$: we decompose $\mathcal{N}(u_n)-\mathcal{N}(u'_n)$ as in
\eqref{fi73}, and combine Lemma \ref{le1}, Lemma \ref{le2}, Lemma
\ref{n1}, Lemma \ref{n4}, and the uniform bound \eqref{fi5}. The
proof of \eqref{fi81} is similar to the proof of \eqref{fi20}. This
completes the proof of Theorem \ref{t2}.

\noindent{\bf Acknowledgment.} This work was finished under the
patient advising of Prof. Carlos E. Kenig while the author was
visiting the Department of Mathematics at the University of Chicago
under the auspices of China Scholarship Council. The author is
deeply indebted to Prof. Kenig for the many encouragements and
precious advices. The author is also very grateful to Prof. Baoxiang
Wang for encouragements and supports.

\end{document}